\documentclass[11pt]{amsart}

\usepackage{amsfonts,amsmath,latexsym,amssymb,verbatim,amsbsy,amsthm}
\usepackage{esint}

\usepackage{graphicx}
\usepackage{caption}
\usepackage{float}
\usepackage{wrapfig}
\usepackage{tikz}
\usepackage{amssymb}
\usepackage{amsmath}
\usepackage{enumitem}
\usepackage{esint}
\usepackage{dsfont}
\usepackage[colorlinks=true, pdfstartview=FitV, linkcolor=blue,citecolor=blue, urlcolor=blue]{hyperref}

\usepackage[top=1in, bottom=1in, left=1in, right=1in]{geometry}

\usepackage{wrapfig}

\theoremstyle{plain}
\newtheorem{THEOREM}{Theorem}[section]

\newtheorem{theorem}[THEOREM]{Theorem}

\newtheorem{lemma}[THEOREM]{Lemma}
\newtheorem{proposition}[THEOREM]{Proposition}

\theoremstyle{definition}

\theoremstyle{remark}

\newcommand{\thm}[1]{Theorem~\ref{#1}}
\newcommand{\lem}[1]{Lemma~\ref{#1}}

\newcommand{\prop}[1]{Proposition~\ref{#1}}

\newcommand{\sect}[1]{Section~\ref{#1}}

\DeclareMathOperator{\Tr}{Tr} %
\def \a {\alpha}
\def \b {\beta}
\def \g {\gamma}
\def \d {\delta}
\def \k {\kappa}
\def \e {\varepsilon}

\def \l {\lambda}
\def \n {\nabla}

\def \t {\tau}
\def \th {\theta}

\def \L {\Lambda}

\def \O {\Omega}

\def \be {{\bf e}}

\def \bj {{\bf j}}

\def \cA {\mathcal{A}}

\def \cC {\mathcal{C}}

\def \cE {\mathcal{E}}

\def \cL {\mathcal{L}}

\def \cN {\mathcal{N}}

\def \cR {\mathcal{R}}

\def \cT {\mathcal{T}}

\usepackage{amsmath}

\def \dH {\dot{H}}\def \dW {\dot{W}}
\def \rmin{\underline{\rho}}
\def \rmax{\overline{\rho}}

\newcommand{\N}{\ensuremath{\mathbb{N}}}   
\newcommand{\Z}{\ensuremath{\mathbb{Z}}}   
\newcommand{\R}{\ensuremath{\mathbb{R}}}   
\newcommand{\T}{\ensuremath{\mathbb{T}}}   

\def \loc {\mathrm{loc}}

\def \one {{\mathds{1}}}

\def \lan {\langle}
\def \ran {\rangle}

\def \p {\partial}

\def \dx  {\, \mbox{d}x}
\def \dxi  {\, \mbox{d}\xi}

\def \dy  {\, \mbox{d}y}
\def \dz  {\, \mbox{d}z}

\def \ds  {\, \mbox{d}s}
\def \dw  {\, \mbox{d}w}
\def \dth  {\, \mbox{d}\th}

\def \dl  {\, \mbox{d}\l}

\def \ddt  {\frac{\mbox{d\,\,}}{\mbox{d}t}}

\begin{document}

\title[Topological diffusion models]{Global solutions to multi-dimensional topological \\ Euler alignment systems}	

\author{Daniel Lear}

\author{David N. Reynolds}

\author{Roman Shvydkoy}

\address{Department of Mathematics, Statistics and Computer Science, University of Illinois at Chicago, 60607}

\email{lear@uic.edu}
\email{dreyno8@uic.edu}
\email{shvydkoy@uic.edu}


\date{\today}


\thanks{\textbf{Acknowledgment.}  
	The work of RS was  supported in part by NSF
	grants DMS-1813351 and DMS-2107956.}

\begin{abstract}
We present a systematic approach to regularity theory of the multi-dimensional Euler alignment systems with topological diffusion introduced in \cite{STtopo}. While these systems exhibit  flocking behavior emerging from purely local communication, bearing direct relevance to empirical field studies, global and even local well-posedness has proved to be a major challenge in multi-dimensional settings due to the presence of topological effects. In this paper we reveal two important classes of global smooth solutions -- parallel shear flocks with incompressible velocity and stationary density profile, and nearly aligned flocks with close to constant velocity field but arbitrary density distribution. Existence of such classes is established via an efficient continuation criterion requiring control only on the Lipschitz norm of state quantities, which makes it accessible to the applications of fractional parabolic theory. The criterion presents a major improvement over the existing result of \cite{RS2020}, and is proved with the use of  quartic paraproduct estimates.
\end{abstract}

\maketitle

\section{Introduction}

One of the major problems of the mathematical theory of systems with self-organization is to understand how global phenomena emerge from local interactions between agents.  In the context of alignment dynamics such questions were addressed already in the seminal works of Cucker and Smale \cite{CS2007a,CS2007b} and studied extensively in \cite{HKPZ-2019,HL2009,MPT2019,MT2014,T2021}.  Two underlying mechanisms lie behind rigorous results in this direction -- either sufficiently strong communication at long range to ensure that no agent escapes the influence of the crowd, or  perpetual connectivity of the flock at local communication range which ensures exchange of information between agents at all times.

We study alignment dynamics of large systems that are described (in the bulk) by the periodic solutions to the hydrodynamic Euler alignment system given by 
\begin{equation}\label{e:main}
\left\{
\begin{split}
\rho_t + \n \cdot (\rho u) & = 0, \\
u_t + u \cdot \n u &= \int_{\T^n}\phi(x,y)(u(y,t) - u(x,t))  \rho(y,t)\dy.
\end{split} \right. 
\end{equation}
Here, $\phi$ stands for the communication protocol, $\rho$ the density of the crowd and $u$ is the velocity field. When $\phi$ is purely local, i.e.
\[
\l \one_{|x-y| <r_0/2} \leq \phi(x,y) \leq  \frac{\L \one_{|x-y| <r_0}}{|x-y|^{\b}}, \quad \b \geq 0,
\]
the mechanism of long range communication is not available. So,
establishing alignment of solutions, which means vanishing of the velocity amplitude,
\[
\cA(t) = \max_{x,y\in \T^n} |u(x,t) - u(y,t)| \to 0, \quad t \to \infty,
\]
would normally rely on sufficiently strong hydrodynamic connectivity expressed by the lower bound on the density  $\rho(\cdot,t) \gtrsim \frac{1}{\sqrt{1+t}}$,
  see \cite{STtopo} and also  \cite{T2021}.  Such a power lower bound is not available for general non-vacuous solutions, except for the case of global singular metric kernels in 1D, or under certain threshold condition in smooth kernel case, see \cite{DKRT2018,ST1,ST2,ST3}.

In \cite{STtopo} a new class of local models was introduced. The underlying principle behind them mirrors observations from the empirical studies of flocks \cite{Bal2008,CCGPS2012,Cavagna2010} which postulate that the probe horizon of given agent $x$ is determined, less by the Euclidean distance, and more by the density of the crowd around $x$. Thus, in thicker crowds communication spreads slower than in thinner ones. Communication based on the density is called {\em topological}, as opposed to {\em metric} which is based on the Euclidean distance. To measure topological distance we implement a metric determined by the mass of a communication domain $\O(x,y)$ between $x$ and $y$:
\begin{equation*}\label{}
d(x,y) = \left(\int_{\O(x,y)} \rho(\xi,t) \dxi \right)^{1/n}.
\end{equation*}
The domain is assumed to be symmetric $\O(x,y) = \O(y,x)$, and obtained by translation and dilatation of a basic domain $\O_0 = \O(0,\be_1)$, which is smooth everywhere but $x,y$ and fits within the intersection of cones of opening $<\pi$ at $x$ and $y$.
Then $\phi(x,y)$ is defined to be a symmetric singular kernel of degree $0<\a<2$ that protocols both topological and metric communications gauged by a parameter $\t>0$,
\begin{equation}\label{e:kernel}
\phi(x,y)=\frac{h(x-y)}{|x-y|^{n+\alpha-\tau}d^{\tau}(x,y)}, \quad 0<\a<2.
\end{equation}
Here, $h = h(r)$ is a smooth local radial bump function satisfying
\[
\l \one_{r<r_0/2} \leq h(r) \leq \L \one_{r<r_0}.
\]
Note that for the metric case $\t = 0$ the action of the kernel is that of the classical (short-range) fractional Laplacian.

The main result of \cite{STtopo} states that with implemented topological protocol  the level of hydrodynamic connectivity can be lowered to $\rho(\cdot,t) \gtrsim \frac{1}{1+t}$ to achieve alignment as long as $\t \geq n$:
\[
\cA(t) \lesssim \frac{1}{\sqrt{\ln (1+t)}}.
\]
Moreover, it is proved that the connectivity condition holds automatically for any non-vacuous data in 1D.

Such a clear advantage of the topological communication  prompts development of regularity theory for the corresponding system \eqref{e:main}-\eqref{e:kernel}. In 1D global well-posedness of the system was established in the same work \cite{STtopo} in classes  
\begin{equation*}\label{e:class}
u \in H^{m+1}, \quad \rho \in H^{m+\a}, \quad m\geq 3,
\end{equation*}
with an adaptation of the De Giorgi method to non-symmetric settings (although $\phi$ itself is symmetric, the total density weighted kernel $\rho(y) \phi(x,y)$ is not).
The major difference between 1D and multi-D cases is manifested in an additional conservation law
\begin{equation}\label{e:eintro}
\begin{split}
e =  u_x + \cL_\phi \rho, \quad e_t +  (u\rho)_x = 0, \\
\cL_\phi f (x) = p.v. \int_{\T^n} \phi(x,y) (f(y) - f(x)) \dy,
\end{split}
\end{equation}
which is present in 1D but not in higher dimensions. This law facilitates a priori control over the higher order norms via the bound $\n^k e \lesssim \n^k \rho$. Yet, for the metric ($\t = 0$) multi-dimensional systems several classes of global smooth solutions have been discovered. Those  include unidirectional flocks \cite{LSuni}, nearly aligned flocks \cite{Snearly}, and flocks with a small spectral gap data in 2D \cite{HeT2017}.  Moreover, an effective continuation criterion proved in \cite{LSuni} requires only control on the gradients $\| \n \rho\|_{L^\infty([0,T)\times \T^n)}$ and $ \| \n u\|_{L^\infty([0,T)\times \T^n)}$. We refer to the text \cite{Sbook} for a systematic exposition.

A similar theory for topological systems has been elusive due to additional topological terms that enter into the conservation law \eqref{e:eintro}, see \eqref{e:etopo} below. Local existence and uniqueness of solutions in class \eqref{e:class} for a large $m\in \N$ was established in \cite{RS2020} by a direct and technical application of fractional energy estimates. This result comes with a rough continuation criterion in terms of $u,\rho \in C^2$. With more effort the technique gives $u,\rho \in C^\g$, where $\g = 1+\e$ for $\a\leq 1$, and $\g = \a+\e$ for $\a>1$. In either case such a criterion is less effective from the application standpoint as the typical fractional parabolic regularity theory, for example, provides Schauder estimates in $C^{1+\d}$ for an indeterminate small  parameter $\d>0$, see \cite{DJZ2018,IJS2018,SS2016} and references therein.

In this work we aim to provide a systematic treatment of the regularity theory for topological systems. Our first result gives the continuation criterion, a direct analogue of the metric case \cite{LSuni}.

\begin{theorem}\label{t:criterion}
Let $(u,\rho) \in L^\infty_\loc([0,T);H^{m+1} \times H^{m+\a})$ be a local non-vacuous solution to \eqref{e:main}-\eqref{e:kernel} such that 
\begin{equation}\label{e:criterion}
 \| \n \rho\|_{L^\infty([0,T)\times \T^n)}+ \| \n u\|_{L^\infty([0,T)\times \T^n)}<\infty.
\end{equation}
Then the solution can be extended beyond the interval $[0,T]$.
\end{theorem}

Our basic methodology is similar to \cite{RS2020} but much different in two technical aspects. First, we establish new sets of estimates for multilinear singular integral operators, see Lemmas \ref{l:subcrit}, \ref{l:subcrit2}, that allow to extract the gradients of $u$ and $\rho$ in all subcritical terms appearing in the energy estimates. We also establish sharp coercivity bounds for the topological diffusion operator
\[
\|\cL_\phi \rho \|_{H^m} \sim \|\rho \|_{H^{m+\a}},
\]
up to a polynomial factor depending only on the gradients as well. Second aspect has to do with the critical terms that emerge from the topological component of the $e$-equation in the form of quartic products. To achieve control on such terms under the criterion assumption \eqref{e:criterion} we employ paraproduct estimates that are inspired by the proof of the positive side of the Onsager conjecture for incompressible Euler equation \cite{CCFS2008}, see \cite{BV2021} for the full overview of this subject. 

\thm{t:criterion} applies to reveal two new classes of global smooth solutions. First is the class of {\em parallel shear flocks} -- similar to unidirectional class discussed in \cite{LSuni} but with incompressible velocity field
\[
u = (U(x_2,\dots,x_n,t),0,\dots,0),
\]
and stationary smooth density $\rho = \rho_0(x_2,\dots,x_n)$. In this case the velocity component $u$ satisfies  a fractional parabolic equation  which will be shown to fulfill the assumptions of known regularity results in \cite{IJS2018,SS2016} that give uniform Schauder bounds on $u$ in class $C^{1+\g}$. Hence, the criterion applies, see \sect{s:psf}.  

Second is the class of {\em nearly aligned flocks} -- solutions with initial velocity amplitude $\cA_0$ inversely proportional to the size of the Sobolev norms of the data and lower bound on the density, see \thm{t:naf} for precise formulation. These solutions may still have large density profiles, so the class can be viewed as a partially small data. The metric case was treated in \cite{Snearly}, see also \cite{DMPW2019} for a similar result in Besov settings. We prove that nearly aligned flocks settle exponentially fast to a traveling wave solution with constant velocity
\[
u \to \bar{u}, \quad \rho \to \rho_\infty(x-t \bar{u}).
\]

Although the results of this present work brings the state of the regularity theory for topological models \eqref{e:main} essentially to the same level of development as for metric ones,  it has to be noted that general global well-posendess in both cases remains an outstanding open problem.

\section{Preliminaries}

In this section we collect all basic properties of the system as well as recall analytical tools, notations, and conventions that will be used in later sections.

\subsection{Notation}
First, we denote all $L^p$-norms by $\|\cdot \|_p$ for short.  The notation $A \lesssim B$ means $A\leq CB$, where $C$ depends only upon absolute constants or a priori bounded quantities such as $\| \n \rho\|_{L^\infty([0,T)\times \T^n)}$ and $ \| \n u\|_{L^\infty([0,T)\times \T^n)}$. $A\sim B$ means $A\lesssim B$ and $B \lesssim A$.

We denote finite differences by
\[
\delta_z f(x)=f(x+z)-f(x), \qquad \d^2_z f(x) = f(x+z) + f(x-z) - 2 f(x).
\]

\subsection{Sobolev spaces}
For $1\leq p < \infty$ and $0<s<1$ we adopt the use of the Gagliardo-Sobolevskii fractional Sobolev spaces
\begin{equation*}\label{}
\| f\|_{W^{s,p}}^p = \|f\|_p^p + \int_{\T^{2n}} \frac{|\d_z f(x)|^p}{|z|^{n+ sp}} h(z) \dz \dx.
\end{equation*}
For the upper range $1<s<2$ one has to use the next Taylor term:
\begin{equation*}\label{e:Sob12}
\| f\|_{W^{s,p}}^p = \|f\|_p^p + \int_{\T^{2n}} \frac{|\d_z f(x) - z \n f(x) |^p}{|z|^{n+ sp}} h(z) \dz \dx.
\end{equation*}
And for the extended range $0<s<2$ including the integer value $s=1$ one can define the Sobolev space using second finite difference: 
\begin{equation*}\label{}
\| f\|_{W^{s,p}}^p = \|f\|_p^p + \int_{\T^{2n}} \frac{|\d_z^2 f(x) |^p}{|z|^{n+ sp}} h(z) \dz \dx.
\end{equation*}
The $L^2$-based spaces will be denoted by $H^s = W^{2,s}$.

In the course of the proof we encounter finite differences with respect to a parameter $\xi$ depending on $z$ and the communication domain at question. Let us recall from \cite{RS2020} the change of variables
\[
\int_{\O(0,z)} f(\xi) \dxi = |z|^{n-1} \int_{\p \O_0} f( |z| U_z \th) \dth,
\]
where $\O_0 = \O(0,\bf{e}_1)$ is the basic communication domain connecting the origin with the first basis vector $\bf{e}_1$, and $U_z: \R^n \to \R^n$ is a unitary transformation sending $\bf{e}_1$ to $z / |z|$, and hence $ \O_0 \to \O(0,z)$. We often keep the same notation $\xi =  |z| U_z \th$ for the variable of integration keeping in mind that $\xi  = \xi(z,\th)$.  We have the following inequality for any function $\xi$ satisfying $|\xi|\leq |z|$:
\begin{equation} \label{e:Sobxi1}
 \int_{\mathbb{T}^{2n}}  \frac{ \left| f(x+ \xi) -  f (x) \right|^2 }{|z|^{n+2s} } h(z)   \dz \dx \lesssim \| f\|_{H^s}^2,  \quad 0<s<1.
\end{equation}
Indeed, by the Parseval identity,
\begin{equation*}\label{}
\begin{split}
\int_{\mathbb{T}^{2n}}  \frac{ \left| f(x+ \xi) -  f (x) \right|^2 }{|z|^{n+2s} } h(z)   \dz \dx = \sum_{k\in \Z^n} |\hat{f}(k)|^2 \int_{\T^{n}} |e^{i \xi\cdot k} - 1|^2 \frac{h(z)}{|z|^{n+2s} }   \dz .
\end{split}
\end{equation*}
Given that $|e^{i \xi\cdot k} - 1|^2 \lesssim \min\{ 1, |z|^2|k|^2\}$,  the integral in $z$ is bounded by $|k|^{2s}$ and the result follows.

\subsection{ Basic properties of the system. Local well-posedness}

A detailed discussion of the properties of the system \eqref{e:main}--\eqref{e:kernel} is presented in \cite{STtopo}. We recall a few that are needed for our future analysis. First, any smooth solution obeys the maximum principle
\begin{equation*}\label{}
\cA(t) \leq \cA_0.
\end{equation*}
The system is invariant under Galilean transformation 
\[
x \to x- t \bar{u}, \quad u \to u -  \bar{u}.
\]
Due to continuity and symmetry of the kernel, solutions preserve mass and momentum
\[
M = \int_{\T^n} \rho \dx, \quad P = \int_{\T^n} \rho u \dx.
\]
In view of the above the mean velocity $\bar{u} = P/M$ is preserved, and we can assume that $\bar{u} = 0$ by modding it out.

We have the following energy law for smooth solutions
\begin{equation}\label{e:energy}
\begin{split}
\cE &= \frac12 \int_{\T^n} \rho |u|^2 \dx , \\
\ddt \cE &= - \int_{\T^{2n}} \rho(x)\rho(y) | u(x) - u(y)|^2 \phi(x,y) \dy \dx.
\end{split}
\end{equation}

Finally, we recall the local well-posedness result proved in \cite{RS2020}.

\begin{theorem}\label{t:main} Let $0<\a < 2$ and $\tau \geq 0$. For any initial data $u_0\in H^{m+1}(\T^n)$, $\rho_0 \in H^{m+\a}(\T^n)$,  $m\geq m(\a,n)$, with no vacuum $\rho_0(x) >0$ there exists a unique non-vacuous solution to the system \eqref{e:main}-\eqref{e:kernel} on  a time interval $[0,T)$ where $T$ depends on the initial conditions,  in the class
	\begin{equation}\label{e:class}
		\begin{split}
		u & \in C_w([0,T), H^{m+1}) \cap L^2([0,T), H^{m+1+\frac{\a}{2}}), \\
		\rho& \in C_w([0,T), H^{m+\a}).
		\end{split}
	\end{equation}
\end{theorem}

Here, $C_w$  stands for weakly continuous functions.

\subsection{Gagliardo-Nirenberg inequalities}
For a function $f \in H^{{s+1}} \cap W^{1,q}$, $s\geq 0$, recall the classical Gagliardo-Nirenberg inequalities
\begin{equation*}\label{ }
\| f\|_{W^{j+1,p}} \leq \| f\|_{H^{s+1}}^\th \| \n f \|_{q}^{1-\th},
\end{equation*}
where 
\begin{equation}\label{e:GNth}
\frac1p = \frac{j}{n} + \left( \frac12 - \frac{s}{n} \right) \th + \frac{1-\th}{q}, \qquad \frac{j}{s} \leq \th \leq 1.
\end{equation}
We will be interested in placing the smallest possible power $\th = \frac{j}{s}$ onto the highest norm $H^{s+1}$ without care about the resulting $q$, because eventually we simply replace $\| \n f\|_{q} \leq \|\n f \|_{\infty}$, which under our assumption will always be bounded a priori.  However, one still needs to ensure that  such a $q$ exists within the allowed range $1\leq q \leq \infty$. For this purpose let us set $\th = \frac{j}{s}$ in \eqref{e:GNth} and obtain
\[
\frac1p = \frac{j}{2s} + \frac1q \left( 1 -  \frac{j}{s}\right) .
\]
Consequently, such a $q$ exists if and only if 
\begin{equation}\label{e:GNp}
   \frac{j}{2s} \leq \frac1p \leq 1 -   \frac{j}{2s},
\end{equation}
and we have (adopting the convention for $\lesssim$)
\begin{equation}\label{e:GN}
\| f\|_{W^{j+1,p}} \lesssim \| f\|_{H^{s+1}}^{\frac{j}{s}}.
\end{equation}
In all the situations we encounter, $p \geq 2$, so  the right hand side of \eqref{e:GNp} will be automatically satisfied.

\subsection{Paraproducts}
In the product estimates of the $e$-equation we will utilize  paraproduct decompositions.  The classical Littlewood-Paley decomposition is given by the series 
\[
f = \sum_{q = 0}^\infty f_q,
\]
where $f_q$ denotes the Littlewood-Paley projection onto the $q$th dyadic shell in Fourier space, see \cite{Grafakos}.   For any $q\in \N$ we also denote
\[
f_{<q} = \sum_{p<q} f_p, \qquad f_{\sim q} = \sum_{q-2\leq p\leq q+2} f_p.
\]
Let us denote the frequency parameters by $\l_q = 2^q$. Recall that 
\[
\|f\|_{H^s}^2 \sim \sum_q \l_q^{2s}\|f_q\|_2^2, \qquad s\geq 0.
\]
Any triple product  can be decomposed into the Bony paraproduct formula:
\[
\lan f,g,h \ran: = \int_{\T^n} f g h \dx = LHH + HLH + HHL,
\]
where 
\begin{equation*}\label{}
\begin{split}
LHH & = \sum_q   \sum_{p > q-1} \lan f_q,  g_{\sim p}, h_{\sim p} \ran ,\\
HLH & = \sum_q \lan f_q, g_{<q} , h_{\sim q} \ran ,\\
HHL & = \sum_q  \lan f_q , g_{\sim q}, h_{<q} \ran.
\end{split}
\end{equation*}
We will encounter further decompositions into quartic paraproducts if one of the terms is a product of two functions, $h = h' h''$. For that purpose we note two identities
\begin{align}
( h' h'')_{<q} & = (h'_{< q+2} h''_{<q+2})_{<q} + \sum_{r >q+1} (h'_{\sim r} h''_{\sim r})_{<q} \label{e:LL},\\
(h' h'')_{\sim q} & = (h'_{<q} \, h''_{\sim q})_{\sim q}+(h'_{\sim q} \, h''_{< q})_{\sim q}+ \sum_{r >q-2} (h'_{\sim r} \, h''_{\sim r})_{\sim q}.\label{e:HH}
\end{align}

Finally, we recall the classical commutator estimate which we will use repeatedly,
\begin{equation}\label{e:classcomm}
\| \p^{m}(fg) - f \p^{m} g \|_2 \leq \|\n f \|_\infty \|g\|_{\dH^{m-1}} + \| f\|_{\dH^{m}} \| g \|_\infty,
\end{equation}
and the product formula
\begin{equation}\label{e:prodHm}
\| \p^m (f g) \|_2 \leq \| f\|_{H^m} \|g\|_\infty + \|f\|_\infty \|g\|_{H^m},
\end{equation}
both of which can be easily obtained via the Bony decomposition.

\subsection{The Faa di Bruno formula}
We will make  repeated use of the Faa di Bruno expansion formula for a multiple derivative of a composite function
\begin{equation*}\label{e:FdB0}
 \partial^{P} h(g)  =\sum_{\bj} \frac{P!}{j_1!1!^{j_1}j_2!2!^{j_2}...j_{P}!P!^{j_{P}}}h^{(j_1+...+j_{P})}(g)\prod_{k=1}^{P} \left(\partial^k g\right)^{j_k},
\end{equation*}
where the sum is over all $P$-tuples of non-negative integers $\bj=(j_1,...,j_{P})$ satisfying
\begin{align*}\label{e:jml}
1j_1+2j_2+...+P j_{P}=P.
\end{align*}
More often we will not need to know the breakdown of repeated derivatives in the product as long as the total order adds up to $P$:
\begin{equation}\label{e:FdB}
 \partial^{P} h(g)  =\sum_{\bj} C_{\bj} h^{|\bj|}(g)\prod_{i=1}^{|\bj|} \partial^{k_i} g, \qquad k_1+\dots+k_{|\bj|} = P.
\end{equation}

\subsection{Subcritical product estimates}

In what follows we encounter many subcritical terms which take on the standard forms described in the lemmas below.

\begin{lemma}\label{l:subcrit} 
Consider the singular integral 
\begin{equation*}\label{e:I}
I(x) =  \int_{\T^n} | \p^{l_1}g_1(x+\eta_1) \ldots  \p^{l_M} g_{M}(x+\eta_M)|  |\d_{\xi_1} \p^{k_1}f_1(x) \ldots \d_{\xi_{N}} \p^{k_{N}} f_{N}(x)| \frac{h(z) }{|z|^{n+\a-2+N} } \dz,
\end{equation*}
where $|\xi_i(z)| \leq | z|$, $\eta_j = \eta_j(z)$, and
\begin{equation}\label{e:klj}
l_1+\ldots+l_M+k_1+\ldots +k_{N} \leq m, \quad N, M \geq 0.
\end{equation}
Then
\begin{equation*}\label{ }
\| I\|_{2} \leq C \prod_{j=1}^M  \|g_j \|^{\a_j}_{H^{m+\a}} \times \prod_{i=1}^N  \|f_i \|^{\b_i}_{H^{m+\a}},
\end{equation*}
for some $\a_j,\b_i \geq 0$,
\[
\a_1+\ldots + \a_M + \b_1 +\ldots+\b_N <1,
\]
where $C$ depends only on the Lipschitz norms of $g_j$'s and $f_i$'s.
\end{lemma}
\begin{proof} We can assume without loss of generality that all $k_i >0$. Indeed, for those that are equal to zero, we replace the finite difference by the gradient thereby lowering $N$, still satisfying \eqref{e:klj}. 

If after this $N=M=0$, then the lemma is trivial. Otherwise,  if we still have $\a - 2 + N <0$, then the singularity is integrable, and we estimate
\[
|I(x)|^2 \leq  C  \int_{\T^n} | \p^{l_1}g_1(x+\eta_1) \ldots  \p^{l_M} g_{M}(x+\eta_M)|^2|\d_{\xi_1} \p^{k_1}f_1(x) \ldots \d_{\xi_{N}} \p^{k_{N}} f_{N}(x)|^2 \frac{h(z) }{|z|^{n+\a-2+N} } \dz.
\]
Let us define 
\[
p_i = \frac{2m}{k_i}, \qquad q_j = \frac{2m}{l_j}.
\]
Then by the H\"older ineqiality,
\[
\| I\|_2 \leq \| g_1\|_{W^{l_1, q_1}} \ldots \| g_M\|_{W^{l_M, q_M}} \times \| f_1\|_{W^{k_1, p_1}} \ldots \| f_N\|_{W^{k_N, p_N}}.
\]
Applying the Gagliardo-Nirenberg inequality \eqref{e:GN} we obtain
\[
\| I\|_2 \leq \| g_1\|^{\a_1}_{H^{m+\a}} \ldots \| g_M\|^{\a_M}_{H^{m+\a}} \times\| f_1\|^{\b_1}_{H^{m+\a}} \ldots \| f_N\|^{\b_N}_{H^{m+\a}},
\]
where
\[
\a_j =  \frac{l_j-1}{m+\a-1}, \quad \b_i = \frac{k_i-1}{m+\a-1}.
\]
These exponents add up to $\frac{m-N-M}{m+\a - 1} <1$,  as desired, because $N+M \geq 1$.

If $\a - 2 +N \geq 0$, let us fix a small $\d>0$ to be determined later and define
\[
s = \frac{\a-2 + N + 2\d }{N}.
\]
If $\d$ is small enough this exponent satisfies $0<s<1$. Let us now distribute the singularity as follows
\[
| I(x) | =  \int_{\T^n} \frac{|\d_{\xi_1} \p^{k_1}f_1(x)|}{|z|^{\frac{n}{p_1} + s}} \dots  \frac{|\d_{\xi_N} \p^{k_N}f_N(x)|}{|z|^{\frac{n}{p_N} + s}} \frac{| \p^{l_1}g_1(x+\eta_1)|}{ |z|^{ \frac{n}{q_1} - \frac{\d}{M}}} \ldots \frac{| \p^{l_M}g_M(x+\eta_M)|}{ |z|^{ \frac{n}{q_M} - \frac{\d}{M}}}\frac{h(z) }{|z|^{\frac{n}{2} - \d}} \dz,
\]
and as before apply the H\"older inequality,
\begin{multline*}
| I(x) | \leq \left(\int_{\T^n} \frac{|\d_{\xi_1} \p^{k_1}f_1(x)|^{p_1}}{|z|^{n+ sp_1}} h(z) \dz \right)^{\frac{1}{p_1}} \dots    \left(\int_{\T^n} \frac{|\d_{\xi_N} \p^{k_N}f_N(x)|^{p_N}}{|z|^{n + sp_N}}  h(z) \dz \right)^{\frac{1}{p_N}} \\ \times
\left(\int_{\T^n} \frac{| \p^{l_1}g_1(x+\eta_1)|^{q_1}}{|z|^{n-\d q_1}} h(z) \dz \right)^{\frac{1}{q_1}}  \dots \left(\int_{\T^n} \frac{| \p^{l_M}g_M(x+\eta_M)|^{q_M}}{|z|^{n-\d q_M}} h(z) \dz \right)^{\frac{1}{q_M}} .
\end{multline*}
Thus,
\begin{equation*}\label{}
\begin{split}
\| I\|_2 & \leq  \| g_1\|_{W^{l_1, q_1}} \ldots \| g_M\|_{W^{l_M, q_M}} \times  \| f_1\|_{W^{k_1+s, p_1}} \ldots \| f_N\|_{W^{k_N+s, p_N}}\\
&  \lesssim \| g_1\|^{\a_1}_{H^{m+\a}} \ldots \| g_M\|^{\a_M}_{H^{m+\a}}\| f_1\|^{\b_1}_{H^{m+\a}} \ldots \| f_N\|^{\b_N}_{H^{m+\a}}, 
\end{split}
\end{equation*}
where 
\[
\a_j =  \frac{l_j-1}{m+\a-1}, \quad \b_i = \frac{k_i-1+s}{m+\a-1}.
\]
Clearly, $\a_1 + \dots +\b_N <1$, as desired.
\end{proof}

We can prove a similar estimate for a more singular integral if one of the integrands can absorb $m+1+\frac{\a}{2}$ derivatives.

\begin{lemma}\label{l:subcrit2} 
Consider the singular integral 
\begin{multline*}\label{e:I}
II(x) =  \int_{\T^n} | \p^{l_1}g_1(x+\eta_1) \ldots  \p^{l_M} g_{M}(x+\eta_M)|  |\d_{\xi_1} \p^{k_1}f_1(x) \ldots \d_{\xi_{N}} \p^{k_{N}} f_{N}(x)| \\
\times  |\d_z \p^d u(x)| \frac{h(z) }{|z|^{n+\a-1+N} } \dz,
\end{multline*}
where $|\xi_i(z)| \leq | z|$, $\eta_j = \eta_j(z)$, and
\begin{equation*}\label{e:kljlemma2}
l_1+\ldots+l_M+k_1+\ldots +k_{N} + d \leq m+1, \quad N \geq 0, \quad d\geq 1.
\end{equation*}
Then
\begin{equation*}\label{ }
\| I\|_{2} \leq C \|u\|_{H^{m+1 + \frac{\a}{2}}}^\g \prod_{j=1}^M  \|g_j \|^{\a_j}_{H^{m+\a}} \times \prod_{i=1}^N  \|f_i \|^{\b_i}_{H^{m+\a}},
\end{equation*}
for some $\a_j,\b_i \geq 0$,
\[
\a_1+\ldots + \a_M + \b_1 +\ldots+\b_N + \g <1,
\]
where $C$ depends only on the Lipschitz norms of $u$, $g_j$'s and $f_i$'s.
\end{lemma}
\begin{proof} As before, we can assume without loss of generality that all $k_i >0$, and  $\a - 1 + N \geq 0$.

Let us fix two small parameters $\d' \ll \d$ so that $\a+ \d < \frac{\a}{2} + 1$, and define
\[
s = \frac{\a-1 + N+ 2\d' }{N+1}, \quad p_i = \frac{2(m+1)}{k_i}, \quad q_j = \frac{2(m+1)}{l_j}, \quad r= \frac{2(m+1)}{d}.
\]
Let us now distribute the singularity as follows
\begin{multline*}
| II(x) | =  \int_{\T^n}  \frac{| \p^{l_1}g_1(x+\eta_1)|}{ |z|^{ \frac{n}{q_1} - \frac{\d'}{M}}} \ldots \frac{| \p^{l_M}g_M(x+\eta_M)|}{ |z|^{ \frac{n}{q_M} - \frac{\d'}{M}}} \frac{|\d_{\xi_1} \p^{k_1}f_1(x)|}{|z|^{\frac{n}{p_1} + s}} \dots  \\ \dots \frac{|\d_{\xi_N} \p^{k_N}f_N(x)|}{|z|^{\frac{n}{p_N} + s}} \frac{| \d_z \p^{d}u(x)|}{ |z|^{ \frac{n}{r} +s}}\frac{h(z) }{|z|^{\frac{n}{2} - \d'}} \dz,
\end{multline*}
and as before apply the H\"older inequality,
\begin{equation*}\label{}
\begin{split}
| II(x) | & \leq \left(\int_{\T^n} \frac{| \p^{l_1}g_1(x+\eta_1)|^{q_1}}{|z|^{n-\d q_1}} h(z) \dz \right)^{\frac{1}{q_1}}  \dots \left(\int_{\T^n} \frac{| \p^{l_M}g_M(x+\eta_M)|^{q_M}}{|z|^{n-\d q_M}} h(z) \dz \right)^{\frac{1}{q_M}} \\
&\times \left(\int_{\T^n} \frac{|\d_{\xi_1} \p^{k_1}f_1(x)|^{p_1}}{|z|^{n+ sp_1}} h(z) \dz \right)^{\frac{1}{p_1}} \dots    \left(\int_{\T^n} \frac{|\d_{\xi_N} \p^{k_N}f_N(x)|^{p_N}}{|z|^{n + sp_N}}  h(z) \dz \right)^{\frac{1}{p_N}} \\ 
& \times \left(\int_{\T^n} \frac{|\d_{z} \p^{d} u(x)|^{r}}{|z|^{n+ sr}} h(z) \dz \right)^{\frac{1}{r}}.
\end{split}
\end{equation*}
Thus,
\begin{equation*}\label{}
\begin{split}
\| II\|_2 & \leq  \| g_1\|_{W^{l_1, q_1}} \ldots \| g_M\|_{W^{l_M, q_M}} \times  \| f_1\|_{W^{k_1+s, p_1}} \ldots \| f_N\|_{W^{k_N+s, p_N}} \| u \|_{W^{d+s, r}} \\
&  \lesssim \| g_1\|^{\a_1}_{H^{m+\a}} \ldots \| g_M\|^{\a_M}_{H^{m+\a}}\| f_1\|^{\b_1}_{H^{m+\a}} \ldots \| f_N\|^{\b_N}_{H^{m+\a}} \| u \|_{H^{m+\a+\d}}^\g ,
\end{split}
\end{equation*}
where 
\[
\a_j =  \frac{l_j-1}{m+\a-1}, \quad \b_i = \frac{k_i-1+s}{m+\a-1}, \quad \g = \frac{d-1+s}{m+\a-1+\d}.
\]
The sum of all exponents is less than
\[
\frac{m- N}{m+\a-1} + s \left( \frac{N}{m+\a-1} + \frac{1}{m+\a - 1+\d} \right).
\]
It remains to notice that if $\d'$ were $0$ then the above expression would be strictly less than $1$. So, by continuity we can pick a small $\d'>0$ for which the sum is still $<1$.  Thus, $\a_1 + \dots+\b_N +\g <1$, as desired.
\end{proof}

\subsection{Coercivity of the topological diffusion}

The last tool we will need in the proof of the continuation criterion is the coercivity estimate for the topological diffusion
\[
\cL_\phi f (x) = p.v. \int_{\T^n} \d_z f(x) \phi(x,x+z) \dz.
\]
It states a very much intuitive fact that $\cL_\phi$ acts as a derivative of order $\a$.  In view of the highly non-linear dependence on the density in the kernel $\phi$ this fact requires a separate treatment.  An estimate of this sort was already established in \cite{RS2020,STtopo}, however the dependence of residual constants was not traced sharply to the gradients of $\rho$, which is important for our particular application. In this section we present a much different and shorter proof based on \lem{l:subcrit}.

\begin{proposition}\label{p:coerc} For any $\rho\in H^{m+\a}$, $0<\a<2$, we have the following estimates
\begin{equation*}\label{e:coerc}
\begin{split}
	\|	\cL_{\phi} \rho \|_{\dH^m} & \leq 2\rmin^{-\t/n} \| \rho \|_{\dH^{m+\a}} + C_1, \\
		\|	\cL_{\phi} \rho \|_{\dH^m} & \geq \frac12 \rmax^{-\t/n} \| \rho \|_{\dH^{m+\a}} - C_2,
\end{split}
\end{equation*}
where $C_1,C_2$ are constants which depend only on $\rmin,\rmax$, and $\| \n\rho\|_\infty$.
\end{proposition}

\begin{proof}
We start by ``freezing the coefficients'' in the topological part of the kernel:
\[
\cL_\phi \rho  = \rho^{-\t/n} \L_\a \rho + \cR \rho,
\]
where
\begin{equation*}\label{}
\begin{split}
\cR \rho & =  \int_{\T^n}  \d_z  \rho\, R_z \frac{h(z)}{|z|^{n+\a}} \dz, \\
 R_z & =   \frac{1}{\left[\fint_{\O(0,z)} \rho(x+\xi) \dxi\right]^{\frac{\t}{n}}} - \frac{1}{\rho^{\frac{\t}{n}}(x)} ,
\end{split}
\end{equation*}
and $\L_\a$ represents the pure fractional Laplacian with  cutoff $h$.

Then
\begin{equation*}\label{}
\begin{split}
\| \p^m \cL_\phi \rho\|_2 & \leq \| \rho^{-\t/n} \p^m \L_\a \rho\|_2 + \| \p^m (\rho^{-\t/n}  \L_\a \rho) - \rho^{-\t/n} \p^m \L_\a \rho \|_2 + \|\p^m \cR \rho\|_2,\\
\| \p^m \cL_\phi \rho\|_2 & \geq \| \rho^{-\t/n} \p^m \L_\a \rho\|_2 - \| \p^m (\rho^{-\t/n}  \L_\a \rho) - \rho^{-\t/n} \p^m \L_\a \rho \|_2 - \|\p^m \cR \rho\|_2.
\end{split}
\end{equation*}
Clearly,
\[
\| \rho^{-\t/n} \p^m \L_\a \rho\|_2 \sim \|\rho\|_{H^{m+\a}}.
\]
It remains to show that all the other terms are of smaller order. Let us start with the commutator. We have by \eqref{e:classcomm},
\[
\| \p^m (\rho^{-\t/n}  \L_\a \rho) - \rho^{-\t/n} \p^m \L_\a \rho \|_2 \leq \| \n \rho^{-\t/n}\|_\infty  \|\rho\|_{H^{m+\a-1}} + \| \rho^{-\t/n}  \|_{H^m} \|  \L_\a \rho \|_\infty.
\]
Then by interpolation,
\[
 \lesssim \e \|\rho\|_{H^{m+\a}} + C +  \|\rho\|_{H^m} \| \rho \|_{W^{\a+\d,\infty}}.
 \]
Applying Gagliardo-Nirenberg inequalities to the last term we further obtain
\[
\lesssim  \e \|\rho\|_{H^{m+\a}} + C +\|\rho\|_{H^{m+\a}}^{ \frac{m-1}{m+\a -1} + \frac{2(\a+\d-1)}{2(m+\a-1) - n}  }.
\]
One can check that the last exponent is strictly less than $1$ if $\d>0$ is small enough. So,  the whole term is
\[
\lesssim  \e \|\rho\|_{H^{m+\a}} + C.
\]

Let us now turn to the remainder term $\p^m \cR$. Its Leibnitz expansion consists of terms
\begin{equation}\label{e:LR}
 \int_{\T^n} \d_z  \p^l  \rho\, \p^{m-l} R_z  \frac{h(z)}{|z|^{n+\a}} \dz.
\end{equation}
We use the following representation for $R_z$
\[
R_z = \fint_{\O_0}\d_\xi \rho(x) \dth \int_0^1\left(\l \rho(x) + (1-\l)\fint_{\O_0} \rho(x+\xi) \dth \right)^{-1-\tau/n} \dl.
\]
Applying $m-l$ derivatives to $R_z$ and using Leibnitz and  Faa di Bruno formula \eqref{e:FdB} we can see that the expansion will consist of terms (taking the communication integrals outside)
\[
\d_\xi \p^{m-l-p} \rho(x) \prod_{i=1}^{|\bj|} \p^{k_i} \rho(x+\xi_i),
\]
up to a bounded function depending on $\l,\rho$, and where all $|\xi_i | \leq |z|$, and $k_1+\dots+k_{|\bj|} = p$. Thus, the integral \eqref{e:LR} will consist of terms bounded by
\begin{equation*}\label{e:LRaux}
 \int_{\T^n} |\d_z  \p^l  \rho| |\d_\xi \p^{m-l-p} \rho(x)| \left| \prod_{i=1}^{|\bj|} \p^{k_i} \rho(x+\xi_i)\right| \frac{h(z)}{|z|^{n+\a}} \dz.
\end{equation*}
We can see that these integrals fall under the scope of  \lem{l:subcrit} with $N= 2$. This proves that the entire residual term is estimated by
\[
\| \p^m \cR \rho \|_2 \lesssim \|\rho\|_{H^{m+\a}}^\th, \qquad \th <1.
\]
The generalized Young's inequality finishes the proof.
\end{proof}

\subsection{Sobolev norm of $\rho^{-\t/n}$}
The last technical ingredient is the Sobolev bound on the power function of the density.
\begin{lemma}\label{l:powerrho}
We have
\[
\|\rho^{-\t/n} \|_{H^{m+\a}} \lesssim \|\rho \|_{H^{m+\a}} + C(\|\n \rho\|_\infty).
\]
\end{lemma}
\begin{proof}
Without loss of generality we can assume that $\a<1$, for otherwise we simply replace $m$ by $m+1$ and $\a$ by $\a-1$. Forming the finite difference we have
\[
\d_z \p^{m} \rho^{-\t/n}(x) =  \p^{m} (\d_z\rho^{-\t/n}(x)).
\]
Using that 
\[
\d_z\rho^{-\t/n}(x) = \d_z \rho(x) \int_0^1(\l \rho(x+z)+(1-\l) \rho(x))^{-1-\t/n} \dl,
\]
we distribute the $m$ derivatives to obtain terms
\[
\p^l \d_z \rho(x) \int_0^1 \p^{m-l}(\l \rho(x+z)+(1-\l) \rho(x))^{-1-\t/n} \dl.
\]
Using the Faa di Bruno expansion in the latter, we obtain terms that are bounded by
\[
|\p^l \d_z \rho(x)| \prod_{i=1}^{|\bj|} | \p^{k_i} \rho(x+\xi_i)|,
\]
where $\xi = 0$ or $\xi=z$, and $k_1+\dots+k_{|\bj|} = m-l$. So, the norm $\|\rho^{-\t/n} \|^2_{H^{m+\a}}$  is bounded by the terms 
\[
\int_{\T^{2n}} |\p^l \d_z \rho(x)|^2 \prod_{i=1}^{|\bj|} | \p^{k_i} \rho(x+\xi_i)|^2 \frac{\dz}{|z|^{n+2\a}}.
\]
If $|\bj|=0$, then $l=m$, and this gives the classical Sobolev norm of $H^{m+\a}$. Otherwise, if $|\bj|\geq 1$, applying the H\"older inequality, we obtain
\[
\leq \| \rho\|_{W^{l+\a+\d, q}} \prod_{i=1}^{|\bj|} \|\rho\|_{W^{k_i,p_i}},
\]
where
\[
q = \frac{m}{l}, \qquad p_i = \frac{m}{k_i}.
\]
And by the Gagliardo-Nirenberg inequality \eqref{e:GN}, this is bounded by $\|\rho\|_{H^{m+\a}}^\g$ with $\g<1$. This finishes the proof.
\end{proof}

\section{The continuation criterion}
In this section we prove \thm{t:criterion}. Instead of working with the momentum-mass system directly, we replace the density with the $e$-quantity given by 
\[
e=\nabla \cdot u + \mathcal{L}_{\phi}\rho.
\]
This strategy is prompted by the fact that placing $m+\a$ derivatives directly on the continuity equation creates a derivative overload on the velocity, which comes with the order of ${m+\a+1}$, higher than even the order of dissipation. The equation for $e$, instead, is capable to handle this issue due to cancelation of the highest order  terms. Since the natural order of regularity of $e$ is $m$ we thus define the grand quantity
\begin{equation*}\label{e:Ym}
Y_m = \| u\|_{\dH^{m+1}}^2 + \|e\|_{\dH^m}^2.
\end{equation*}

First, let us note that directly from the continuity equation we have a control over the lower and upper bounds on the density:
\[
\rmin = \min \rho, \qquad \rmax = \max \rho.
\]
Differentiating at the maximum and the minimum points we obtain
\begin{equation*}\label{}
\begin{split}
\ddt \rmax & \leq \|\n u\|_\infty \rmax, \\
\ddt \rmin^{-1} & \leq \|\n u\|_\infty \rmin^{-1}.
\end{split}
\end{equation*}
By virtue of the assumption \eqref{e:criterion} these two are uniformly bounded on the interval $[0,T)$. Consequently, the coercivity bounds of \prop{p:coerc} imply that 
\[
Y_m \sim \| u\|_{\dH^{m+1}}^2 +   \|\rho\|_{H^{m+\a}}^2.
\]
Thus controlling $Y_m$ we control the solution in the needed class.

The theorem will follow by the Gr\"onwall inequality if we establish 
\begin{equation*}
\ddt Y_m \leq C(\rmin,\rmax, \|\n \rho\|_{\infty},\|\n u\|_{\infty} )  Y_m,
\end{equation*}
on the interval of regularity $[0,T)$. This will be the main goal of the next two sections.

\subsection{Estimates on the velocity equation}\label{s:u}
The goal of this section is to establish the bound
\begin{align*}
\ddt\|u\|_{\dot{H}^m}^2\leq C Y_m - c_0 \| u\|^2_{H^{m+1 + \frac{\a}{2}}},
\end{align*}
for  $c_0>0$, where $C$ depends on all the norms we already control uniformly on $[0,T)$. 
Let us rewrite the velocity equation as
\[
\begin{split}
u_t + u \cdot \n u & = \cC_\phi(u,\rho), \\
\cC_\phi(u,\rho)(x) & = \int_{\T^n} \phi(x,x+z) \d_z u(x)  \rho(x+z) \dz = \cL_\phi(u\rho) - u \cL_\phi \rho.
\end{split}
\]
Let us apply $\p^{m+1}$ and test with $\p^{m+1} u$. We have 
\[
\p_t \| u\|_{\dH^{m+1}}^2  = - \int_{\T^n} \p^{m+1} (u \cdot \n u) \cdot \p^{m+1}  u \dx + \int_{\T^n}  \p^{m+1} \cC_\phi(u,\rho)\cdot \p^{m+1}  u \dx.
\]
The transport term is estimated using the classical commutator estimate
\[
\p^{m+1} (u \cdot \n u) \cdot \p^{m+1}  u = u \cdot \n  (\p^{m+1} u) \cdot \p^{m+1}  u + [\p^{m+1} , u] \n u \cdot \p^{m+1}  u .
\]
Then
\[
\int_{\T^n} u \cdot \n  (\p^{m+1} u) \cdot \p^{m+1}  u \dx = - \frac12 \int_{\T^n} (\n \cdot u) | \p^{m+1}  u|^2 \dx \leq \|\n u\|_\infty \| u\|_{\dH^{m+1}}^2,
\]
and using  \eqref{e:classcomm} for $f = u$, $g = \n u$, we obtain
\[
\int_{\T^n} |[\p^{m+1} , u] \n u \cdot \p^{m+1}  u  | \dx \leq  \|\n u\|_\infty \| u\|_{\dH^{m+1}}^2.
\]
Thus,
\[
\p_t \| u\|_{\dH^{m+1}}^2  \lesssim Y_m +  \int_{\T^n} \p^{m+1} \cC_\phi(u,\rho)\cdot \p^{m+1}  u \dx.
\]

In the rest of the argument we focus on estimating the commutator term.  So, we expand by the product rule
\begin{equation*}
	\p^{m+1} \cC_\phi(u,\rho) = \sum_{k=k_1+k_2 = 0}^{m+1} \frac{(m+1)!}{k_1!k_2!(m+1 - k)!}\cC_{\p^{m+1-k} \phi}(\p^{k_1} u, \p^{k_2} \rho).
\end{equation*}
Various terms in this expansion will be estimated differently. One special end-point case provides the necessary dissipation.

\subsubsection{Case $k_1 = m+1$} We symmetrize to obtain
\begin{equation*}\label{e:com+1}
\begin{split}
\int_{\T^n} \cC_\phi(\p^{m+1} u,\rho) \cdot \p^{m+1} u \dx & = -\frac12 \int_{\T^{2n}} \rho(x) |\d_z \p^{m+1} u(x)|^2 \phi(x,x+z) \dz \dx \\
&+  \frac12 \int_{\T^{2n}} \d_z \rho(x) \d_z \p^{m+1} u(x) \p^{m+1} u(x) \phi(x,x+z) \dz \dx\\
& \leq  - c  \| u\|_{\dH^{m+1+\frac{\a}{2}}}^2 + \int_{\T^n} |\p^{m+1} u(x)| \int_{\T^n} |\d_z \p^{m+1} u(x)| \frac{h(z)}{|z|^{n+\a -1}} \dz \dx.
\end{split}
\end{equation*}
In the last term the inner integral falls under the scope of \lem{l:subcrit2} with $d=m+1$, $N=0$. So, it applies together with the generalized Young inequality to yield the bound by $C Y_m + \e \| u\|_{\dH^{m+1+\frac{\a}{2}}}^2$. 

\subsubsection{Case $k_2 = m+1$} The other extreme case is when all derivatives fall on the density in the numerator. This causes a derivative overload on $\rho$ at least when $\a<1$.  We therefore apply the following relaxation argument:
\[
\int_{\T^n} \cC_{\phi} ( u, \p^{m+1} \rho) \cdot \p^{m+1} u  \dx  = \int_{\T^{2n}} \phi(x,x+z) \d_z u(x) \p^{m+1} \rho(x+z) \p^{m+1} u(x) \dz \dx.
\]
Observe that 
\[
\p^{m+1} \rho(x+z)  = \p_z \p_x^{m} \rho(x+z)  = \p_z ( \p_x^m \rho(x+z)  - \p_x^m \rho(x)) = \p_z \d_z \p^m \rho(x).
\]
Let us integrate by parts in $z$:
\begin{multline*}
\int_{\T^n} \cC_{\phi} ( u, \p^{m+1} \rho) \cdot \p^{m+1} u  \dx	= \int_{\T^{2n}}  \phi(x,x+z) \p u(x+z)   \d_z \p^m \rho(x) \p^{m+1} u(x) \dz \dx \\+  \int_{\T^{2n}} \p_z \phi(x,x+z) \d_z u(x)  \d_z \p^m \rho(x) \p^{m+1} u(x) \dz \dx 
	 := J_1+ J_2.
\end{multline*}
Symmetrizing in $J_1$ we further split
\begin{multline*}
J_1 =  \int_{\T^{n}} \left(\int_{\T^n} \d_z \p u(x)   \d_z \p^m \rho(x) \phi \dz\right) \p^{m+1} u(x) \dx - \int_{\T^{n}} \left(\int_{\T^n}  \d_z \p^m \rho(x) \d_z \p^{m+1} u(x) \phi \dz \right)  \p u(x)  \dx\\
: = J_{11} + J_{12}.
\end{multline*}
The inner integral of $J_{11}$ falls under the scope of \lem{l:subcrit2} with $d=1$, $N=1$. We thus bound it by $C Y_m + \e \| u\|_{\dH^{m+1+\frac{\a}{2}}}^2$ as we did earlier. As for $J_{12}$ the $\p u(x)$ is simply bounded a priori, so we obtain
\begin{equation}\label{e:J12}
J_{12} \lesssim \int_{\T^{2n}} \frac{|  \d_z \p^m \rho(x)|}{|z|^{ \frac{n}{2} + \frac{\a}{2}}} \frac{| \d_z \p^{m+1} u(x)|}{ |z|^{ \frac{n}{2} + \frac{\a}{2}}} \dz  \dx \leq \|\rho\|_{H^{m+ \frac{\a}{2}}} \|u\|_{H^{m+1+\frac{\a}{2}}} \leq C Y_m + \e \| u\|_{\dH^{m+1+\frac{\a}{2}}}^2.
\end{equation}

As to $J_2$, let us first observe that $\p_z \phi(x,x+z)= \psi(x,x+z) $ is antisymmetric, $\psi(x,y) = -\psi(y,x)$. Consequently, performing symmetrization we obtain
\[
J_2 = \frac12 \int_{\T^{2n}} \p_z \phi(x,x+z) \d_z u(x)  \d_z \p^m \rho(x) \d_z \p^{m+1} u(x) \dz \dx.
\]
Since  
\[
 \p_z \phi(x,x+z) = - \frac{ (n+\a-\t) h(z) z_i}{|z|^{n+\a+2 - \t} d^\t(x,x+z)} + h(z) \frac{\p_z \int_{\O(x,x+z)} \rho(\xi) \dxi}{|z|^{n+\a-\t} d^{\t + n}(x,x+z)} +  \frac{\p_z h(z)}{|z|^{n+\a- \t} d^\t(x,x+z)},
\]
and noticing that
\[
\left| \p_z \int_{\O(x,x+z)} \rho(\xi) \dxi \right| \leq \|\rho\|_\infty |z|^{n-1},
\]
we can see that 
\[
|\p_z \phi(x,x+z)| \lesssim \frac{\one_{|z|<2r_0}}{|z|^{n+\a+1}}.
\]
The one derivative loss is compensated by $|\d_z u(x)|\leq |z| \|\n u\|_{\infty}.$ With this at hand we estimate 
\[
J_2 \lesssim  \int_{\T^{2n}} \frac{|\d_z \p^m \rho(x) \d_z \p^{m+1} u(x)|}{|z|^{n+\a}} \dz \dx ,
\]
and we are back to \eqref{e:J12}.

\subsubsection{All other cases}
The bulk of the other terms can be estimated simultaneously.  We start by the standard symmetrization:
\[
\begin{split}
\int_{\T^n} \cC_{\p^{m+1-k}  \phi} ( \p^{k_1} u, \p^{k_2} \rho) \cdot \p^{m+1} u  \dx & = \frac12 \int_{\T^{2n}} \d_z \p^{k_1} u (x) \d_z \p^{k_2} \rho (x) \p^{m+1} u(x) \p^{m+1-k}  \phi(x,x+z) \dz \dx\\
& + \frac12 \int_{\T^{2n}} \d_z \p^{k_1} u (x) \p^{k_2} \rho (x) \d_z \p^{m+1} u(x) \p^{m+1-k}  \phi(x,x+z) \dz \dx\\
& = J_1 + J_2.
\end{split}
\]

The term $J_1$ is subcritical. To see that we use the
 Faa di Bruno expansion \eqref{e:FdB} for the kernel $ \p^{m+1-k}  \phi(x,x+z) $. Each term takes the form
\[
\int_{\T^{2n}}  \d_z \p^{k_1} u (x) \d_z\p^{k_2} \rho (x) \p^{m+1} u(x) \frac{h(z)}{|z|^{n+\a-\t}} \frac{\prod_{i=1}^{|\bj|}\int_{\Omega(x,x+z)}\partial^{l_i} \rho(\xi)d\xi}{d^{\tau+|\bj| n}(x,x+z)}\dz\dx,
\]
where $l_1+\dots+l_{|\bj|} = m+1-k$. Integrating by parts in each of the topological domains and reducing to the basic $\p \O_0$ we obtain
\[
\int_{\Omega(x,x+z)}\partial^l \rho(\xi)\dxi =|z|^{n-1} \int_{\p\O_0}\d_{\xi}\p^{l-1} \rho(x)\dth,
\]
where $\xi = \xi(\th, z)$, $|\xi|\leq |z|$. Moving all the communication integration outside, we obtain a family of terms
\[
\int_{\T^n} \p^{m+1} u(x)  \int_{\T^n} \d_z \p^{k_1} u (x) \d_z\p^{k_2} \rho (x) \prod_{i=1}^{|\bj|} \d_{\xi_i}\partial^{l_i-1} \rho(x) \frac{h(z)}{|z|^{n+\a-\t}} \frac{|z|^{(n-1)|\bj|}}{d^{\tau+|\bj| n}(x,x+z)}\dz\dx.
\]
Here the interior integral is bounded by
\[
I(x) = \int_{\T^n} |\d_z \p^{k_1} u (x) \d_z\p^{k_2} \rho (x)| \prod_{i=1}^{|\bj|} |\d_{\xi_i}\partial^{l_i-1} \rho(x)| \frac{h(z)}{|z|^{n+\a+|\bj|}} \dz.
\]
The total order of derivatives here is $m+1 - |\bj|$. So, if $|\bj|\geq 1$, then this term falls under \lem{l:subcrit} to produce the necessary estimates in the same fashion  as previously.  If $|\bj| = 0$, this corresponds to the situation when no derivatives fall on the kernel, and we have $k_1+k_2 = m+1$. Then  we are dealing with the terms of type
\[
II(x) = \int_{\T^n} |\d_z \p^{k_1} u (x) \d_z\p^{k_2} \rho (x)| \frac{h(z)}{|z|^{n+\a}} \dz.
\]
Since $k_2 < m+1$, $k_1\geq 1$, these fall under \lem{l:subcrit2} with $d = k_1$ and $N=1$ to conclude the estimate.

 Turning to $J_2$ we can see that the Faa di Bruno expansion of the kernel produces terms of type
 \[
\int_{\T^{2n}}  \d_z \p^{k_1} u (x) \p^{k_2} \rho (x) \d_z\p^{m+1} u(x)  \frac{\prod_{i=1}^{|\bj|}\int_{\Omega(x,x+z)}\partial^{l_i} \rho(\xi)d\xi}{d^{\tau+|\bj| n}(x,x+z)} \frac{h(z)}{|z|^{n+\a-\t}} \dz\dx,
\]
which, moving the communication integrals outside, are bounded by,
\[
\int_{\T^{2n}} | \d_z \p^{k_1} u (x) \p^{k_2} \rho (x) \d_z\p^{m+1} u(x)| \prod_{i=1}^{|\bj|}|\partial^{l_i} \rho(x+\xi_i)|\frac{h(z)}{|z|^{n+\a}} \dz\dx.
\]
Denoting
\[
p_i = \frac{2(m+1)}{l_i}, \qquad q_{1} = \frac{2(m+1)}{k_1}, \qquad q_{2} = \frac{2(m+1)}{k_2},
\]
and distributing the $\a$-exponent between the terms accordingly,
\[
\int_{\T^{2n}} \frac{| \d_z \p^{k_1} u (x)|}{ |z|^{\frac{n}{q_1}+\frac{\a}{2}+\d}}  \frac{|\p^{k_2} \rho (x)|}{ |z|^{\frac{n}{q_2} -\frac{\d}{|\bj|+1}}} \frac{| \d_z\p^{m+1} u(x)|}{ |z|^{\frac{n}{2}+\frac{\a}{2}}}  \prod_{i=1}^{|\bj|} \frac{|\partial^{l_i} \rho(x+\xi_i)|}{ |z|^{\frac{n}{p_i}- \frac{\d}{|\bj|+1}}} h(z)\dz\dx,
\]
we apply the H\"older inequality to obtain
\[
\leq \|u\|_{H^{m+1+\a/2}} \|u\|_{W^{k_1+\frac{\a}{2} + \d,q_1}} \| \rho\|_{W^{k_2,q_2}}  \prod_{i=1}^{|\bj|}  \| \rho\|_{W^{l_i,p_i}}.
\]
By the Gagliardo-Nirenberg inequality \eqref{e:GN} applied to each term except the first one we obtain
\[
\lesssim  \|u\|_{H^{m+1+\a/2}} \|u\|^{\g_1}_{H^{m+\a}} \| \rho\|^{\g_2}_{H^{m+\a}} \prod_{i=1}^{|\bj|}  \| \rho\|^{\b_i}_{H^{m+\a}},
\]
where 
\[
\g_1 = \frac{k_1 + \frac{\a}{2} - 1+\d}{m+\a-1},\quad \g_2 = \frac{k_2  - 1}{m+\a-1}, \quad \b_i = \frac{l_i - 1}{m+\a-1}.
\]
The sum of all these exponents is equal to $\frac{m+\frac{\a}{2} - 1 - |\bj|+\d}{m+\a-1} <1$ provided $\d$ is small enough.  Application of the generalized Young inequality, and noting that $\|u\|^{\g_1}_{H^{m+\a}} \leq \|u\|^{\g_1}_{H^{m+1+\frac{\a}{2}}}$, produces
\[
\lesssim C Y_m + \e \| u\|_{\dH^{m+1+\frac{\a}{2}}}^2.
\]

\subsection{Estimates on the $e$-equation}\label{s:e}

The goal of this section is to show the bound
\begin{align*}
\ddt\|e\|_{\dot{H}^m}^2\leq C Y_m + \e \| u\|^2_{H^{m+1 + \frac{\a}{2}}},
\end{align*}
for any $\e>0$, where $C$ depends on all the norms we already control uniformly on $[0,T)$, and on $\e$.  

Taking the divergence of the momentum equation and using the continuity equation we obtain the following equation on $e$,
\begin{equation}\label{e:etopo}
e_t+\nabla \cdot (ue)=(\nabla \cdot u)^2-\mathrm{Tr}(\nabla u)^2+\cT[\rho,u], 
\end{equation}
where 
\begin{equation*}\label{}
\cT[\rho,u] =  \partial_t(\mathcal{L}_{\phi}(\rho))+\nabla \cdot \mathcal{L}_{\phi}(\rho u).
\end{equation*}

Let us take a closer look the the topological term $\cT$ and work out a more explicit formula for it. We have
\[
\cT[\rho,u] = \mathcal{L}_{\phi}(\rho_t)+\mathcal{L}_{\phi_t}(\rho) + \mathcal{L}_{\phi}(\nabla \cdot (\rho u))+\mathcal{L}_{\nabla \phi \cdot}(\rho u).
\]
The first and third terms obviously cancel by the continuity equation. For the rest we have
\begin{align*}
\mathcal{L}_{\phi_t}(\rho)&=-\frac{\t}{n}\int_{\mathbb{T}^n}\frac{\int_{\O(x,x+z)}\rho_t(\xi) \dxi}{d^{\t+n}(x,x+z)}\d_z\rho(x) \frac{h(z)}{|z|^{n+\a-\t}} \dz\\
&=\frac{\t}{n}\int_{\mathbb{T}^n}\frac{\int_{\O(x,x+z)}\nabla \cdot (\rho u)(\xi) \dxi}{d^{\t+n}(x,x+z)}\d_z\rho(x) \frac{h(z)}{|z|^{n+\a-\t}} \dz,\\
\mathcal{L}_{\nabla\phi \cdot}(\rho u)&=\int_{\mathbb{T}^n}\nabla \phi(x,x+z) \cdot \d_z(\rho u)(x) \dz\\
&=-\frac{\tau}{n}\int_{\mathbb{T}^n}\frac{h(z)}{|z|^{n+\alpha-\tau}}\frac{\int_{\Omega(x,x+z)} \nabla \rho(\xi) \dxi}{d^{\tau+n}(x,x+z)}\cdot \d_z(\rho u)(x) \dz.
\end{align*}
We arrive at
\begin{equation*}\label{}
\cT[\rho,u]  = \frac{\t}{n}\int_{\T^n} \int_{\O(0,z)}[ \nabla \cdot (\rho u)(x+\xi)  \d_z\rho(x) - \n \rho(x+\xi) \cdot \d_z (\rho u)(x) ]   \frac{h(z)}{|z|^{n+\a-\t} d^{\t+n}(x,x+z)} \dxi \dz.
\end{equation*}

Let us now compute the energy equation for  the $H^m$-norm:
\[
\ddt \|e \|_{\dH^m}^2 = - \p^m e\, \p^m \left(u \cdot \n e +e \n\cdot u  \right) + \p^m e \, \p^m\left[(\nabla \cdot u)^2-\mathrm{Tr}(\nabla u)^2\right] + \p^m e \, \p^m \cT[\rho,u].
\]
Estimating the last term will be the main technical component of this section.  So, let us make a few quick comments as to the remaining terms.  The transport term becomes
\[
\p^m e \, (u \cdot \n \p^m e)  +  \p^m e\, \left[\p^m(u \cdot \n e) - u \cdot \n \p^m e\right]+ \p^m e \, \p^m(e \n \cdot u) .
\]
In the first term we integrate by parts and estimate
\[
| \p^m e u \cdot \n \p^m e| \leq \|e \|_{\dH^m}^2 \|\n u\|_\infty \lesssim Y_m.
\]
For the next term we use the  commutator estimate \eqref{e:classcomm} to obtain
\[
|\p^m e [ \p^m( u \cdot \n e) - u \cdot \n \p^m e]| \leq \|\n u\|_\infty \|e \|^2_{\dH^m}  +\|e \|_{\dH^m} \| u  \|_{\dH^m} \|\n e\|_\infty.
\]
Using the Gagliardo-Nirenberg inequality we estimate the latter term as
\[
\|e \|_{\dH^m} \| u  \|_{\dH^m} \|\n e \|_\infty \leq \|e \|_{\dH^m} \| u\|_{\dH^{m+1}}^{\th_1} \|\n u\|_\infty^{1-\th_1} \| e\|_{\dH^m}^{\th_2} \|e\|_\infty^{1-\th_2},
\]
where $\th_1 = \frac{2(m-1)-n}{2m-n}$ and $\th_2 = \frac{2}{2m-n}$. The two exponents add up to $1$, so by the generalized Young inequality,
\[
\leq (\| e\|_{\dH^m}^2 + \| u\|_{\dH^{m+1}}^2) ( \|e\|_\infty +  \|\n u \|_\infty) \leq ( \|e\|_\infty +  \|\n u\|_\infty) Y_m.
\]
Next term in the $e$-equation is estimated by the product formula \eqref{e:prodHm}.
So, we have
\[
| \p^m e  \p^m(e \n \cdot u) | \leq \|e \|_{\dH^m}^2 \|\n u\|_\infty + \|e \|_{\dH^m} \|e\|_\infty  \| u\|_{\dH^{m+1}} \leq ( \|e\|_\infty +  \|\n u\|_\infty) Y_m.
\]
Finally,
\[
| \p^m e [(\n \cdot u)^2 -  \Tr(\n u)^2] | \leq \|e \|_{\dH^m} \|u\|_{\dH^{m+1}}  \|\n u\|_\infty \leq  \|\n u\|_\infty Y_m.
\]
Thus, 
\begin{equation*}
\ddt \|e \|_{\dH^m}^2 \leq ( \|e\|_\infty +  \|\n u\|_\infty) Y_m + \p^m e \, \p^m \cT[\rho,u].
\end{equation*}

Let us address the issues related with the norm $\|e\|_\infty$. For the range $0<\a<1$ the norm is uniformly bounded by a straightforward application of the representation $e=\n\cdot u -\cL_{\phi}(\rho)$,
\[
\|e\|_{\infty}\lesssim \|\n u\|_{\infty} +\|\L^{\a} \rho\|_{\infty} \leq \|\n u\|_{\infty}+\|\n \rho\|_{\infty}.
\]
So, in this case $\|e\|_{\infty}$ is a priori bounded.

To complete the full range $(1\leq \a <2)$ a more subtle estimate is required. Coming back to the transport terms, we just need to consider a more precise computation for the expression:
\begin{equation}\label{badtransportterms}
\|e \|_{\dH^m} \left( \| u  \|_{\dH^m} \|\n e\|_\infty + \|e\|_\infty  \| u\|_{\dH^{m+1}}\right).
\end{equation}
Now, to take advantage of dissipation, we apply the following Gagliardo-Nirenberg inequalities to the terms of the last factor:
\begin{equation*}
  \begin{split}
    \|u\|_{\dot{H}^m} &\lesssim \| u \|_{\dH^{m+1+ \frac{\a}{2}}}^{\frac{2(m-1)}{2m+\a}},\\
	\|u \|_{\dH^{m+1}} & \lesssim  \|u \|_{\dH^{m+1+ \frac{\a}{2}}}^{\frac{2m}{2m+\a}}, \\
  \end{split}
\quad\qquad \quad
  \begin{split}
    \|\n e\|_{\infty}&\leq \|e \|_{\dH^{m }}^{\th_1} \|\L^{-1}e\|_{\dW^{2-(\a+\e),\infty}}^{1-\th_1},\\
	\| e\|_{\infty} &\leq \|e \|_{\dH^{m}}^{\th_2} \|\L^{-1} e\|_{\dW^{1-\left(\frac{\a}{2}-\e\right),\infty}}^{1-\th_2},
  \end{split}
\end{equation*}
where
\[
\th_1=\frac{\a+\e}{m+1-(2-(\a+\e))-\frac{n}{2}}, \qquad \th_2=\frac{\frac{\a}{2}-\e}{m+\frac{\a}{2}-\e-\frac{n}{2}}.
\]
By generalized Young inequality, we further obtain that \eqref{badtransportterms} can be bounded by
\[
\leq \e \| u \|_{\dH^{m+1+ \frac{\a}{2}}}^{2} +C_\e \, p_N(|\n u|_{\infty}, \|\L^{-1} e\|_{W^{2-(\a+\e),\infty}} ) \|e\|_{\dH^{m}}^{\max\{(1+\th_1)q_1, (1+\th_2)q_2\}},
\]
with $q_1, q_2$ conjugate exponents of
\[
p_1=\frac{2m+\a}{m-1}, \qquad p_2=\frac{2m+\a}{m}.
\]
We obtain $(1+\th_1)q_1<2$ as long as $m > 1 +\left(\frac{2+\a}{2-\a}\right)\frac{n}{2}$  and $(1+\th_2)q_2<2$ as long as $m> \frac{n\a}{4\e}.$ 
In addition, as we need that $\frac{\a}{2}+\e<1$ we impose the smallness conditon $2\e<2-\a$, which give us the required bound for the exponent, i.e. $\max\{(1+\th_1)q_1, (1+\th_2)q_2\}<2$ if
\begin{equation*}\label{regularity}
m > 1 +\left(\frac{2+\a}{2-\a}\right)\frac{n}{2}.
\end{equation*}
Using again the definition of the $e$-quantity we have 
$$\|\L^{-1} e\|_{W^{2-(\a+\e),\infty}}=\|\L^{-1} \n u\|_{W^{2-(\a+\e),\infty}}+ \|\L^{-1} \cL_{\phi}\rho\|_{W^{2-(\a+\e),\infty}},$$
where the first term is trivially bounded by $\|\n u \|_{\infty}$ thanks to the fact that $2-(\a+\e)<1$. For the last one we need to work a litle bit. 

To finish with the transport terms we apply  the next result.
\begin{lemma} For $1\leq \a<2$ and $\e>0$ such that $\a+\e<2$ the following bound holds:
\[
\|\L^{-1} \cL_{\phi}\rho\|_{W^{2-(\a+\e),\infty}} \lesssim \| \n \rho\|_{\infty},
\]
where $\lesssim$ means up to a factor of $\rmin, \rmax.$
\end{lemma}
\begin{proof}
The idea of the proof is just to use the smoother properties of the Riesz potential $\L^{-1}$ and $\L^{1-(\a+\e)}.$ In first place, applying  interpolation we have
\[
\|\L^{-1} \cL_{\phi}\rho\|_{W^{2-(\a+\e),\infty}}\leq \|\L^{-1} \cL_{\phi}\rho\|_{\infty} + \|\L^{1-(\a+\e)} \cL_{\phi}\rho\|_{\infty}.
\]
Since $-1<1-(\a+\e)<0$, we focus our attention on the last one. The other follows similar ideas.
By definition of Riesz potential  we have the expression
\[
\L^{1-(\a+\e)} \cL_{\phi}\rho(x)=\int_{\T^n}\frac{\cL_{\phi}\rho(x+y)}{|y|^{n-(1-(\a+\e))}}\dy=\int_{\T^{2n}}\frac{\rho(x+y)-\rho(x+z)}{|y|^{n-(1-(\a+\e))} |y-z|^{n+\a-\t}} \frac{h(|y-z|)}{d_{\rho}^{\tau}(x+y,x+z)}\dz\dy.
\]
Let us get rid of the dependence on $x$ in the topological kernel by freezing the coefficients as we did previously,
\begin{multline*}
\L^{1-(\a+\e)} \cL_{\phi}\rho(x)= \rho^{-\frac{\tau}{n}}(x)\int_{\T^{2n}}\frac{\rho(x+y)-\rho(x+z)}{|y|^{n-(1-(\a+\e))} |y-z|^{n+\a}} h(|y-z|)\dz\dy\\
+\int_{\T^{2n}}\frac{\rho(x+y)-\rho(x+z)}{|y|^{n-(1-(\a+\e))} |y-z|^{n+\a}} h(|y-z|)\times \left( \frac{1}{\left[\fint_{\O(y,z)} \rho(x+\xi) \dxi\right]^{\frac{\t}{n}}} - \frac{1}{\rho^{\frac{\t}{n}}(x)} \right) \dz\dy.
\end{multline*}
For sake of brevity, due to extra cancelation of the last factor which give us $\lesssim |y-z|\|\n \rho \|_{\infty}$, we focus only in the first term, which is the most singular one. To overcome this issue we apply  integration by parts and the fact that $|y-z|^{-(n+\a)}\approx \n_{z}\cdot (y-z) |y-z|^{-(n+\a)}$, which gives us
\begin{align*}
\int_{\T^{2n}}\frac{\rho(x+y)-\rho(x+z)}{|y|^{n-(1-(\a+\e))} |y-z|^{n+\a}} h(|y-z|)\dz\dy &\approx \int_{\T^{2n}}\frac{\n\rho(x+z)}{|y|^{n-(1-(\a+\e))} |y-z|^{n+\a-1}} h(|y-z|)\dz\dy\\
&\approx \|\n \rho\|_{\infty}\int_{\T^{2n}} \frac{1}{|z+w|^{n-(1-(\a+\e))} |w|^{n+\a-1}} \dw\dz.
\end{align*}
Integrating first in the variable $z$ and then in $w$, the last double integral is bounded by an universal constant and consequently we have proved our goal.
\end{proof}

In conclusion,  we have proved the inequality
\begin{equation*}
\ddt \|e \|_{\dH^m}^2 \leq ( \|\n u \|_\infty +  \|\n \rho\|_\infty) Y_m + \p^m e \, \p^m \cT[\rho,u].
\end{equation*}

We now focus solely on the topological term. First let us derive a form of $\cT$ that is most suitable to our analysis. Integrating by parts inside the communication integrals we obtain
\begin{multline*}
 \int_{\O(0,z)}[ \nabla \cdot (\rho u)(x+\xi)  \d_z\rho(x) - \n \rho(x+\xi) \cdot \d_z (\rho u)(x) ]  \dxi \\ =  \int_{\p \O(0,z)}[ (\rho u)(x+\xi)  \d_z\rho(x) - \rho(x+\xi)\d_z (\rho u)(x) ]  \cdot \nu_\xi \dxi,
\end{multline*}
 using cancelation of the integral of a constant,
\[
= \int_{\p \O(0,z)}[ \d_\xi(\rho u)(x)  \d_z\rho(x) - \d_\xi \rho(x)\d_z (\rho u)(x) ]  \cdot \nu_\xi \dxi,
\]
and adding and subtracting cross-difference terms,
\[
= \int_{\p \O(0,z)}[ \d_\xi \rho \d_z \rho \d_\xi u -\d_\xi \rho \d_z \rho \d_z u + \rho \d_z \rho \d_\xi u - \rho \d_\xi \rho \d_z u]  \cdot \nu_\xi \dxi.
\]
Using unitary transformation of $\O(0,z)$ to the basic domain we obtain the following representation (here $\xi = \xi(\th, z) = |z|U_z \th$)
\begin{equation*}\label{}
\begin{split}
\cT  & = \cT_1 -  \cT_2 + \cT_3 - \cT_4, \\
\cT_1 & = \frac{\t}{n}\int_{\p \O_0}\int_{\T^n} \d_\xi \rho \d_z \rho \d_\xi u \frac{h(z) U_z \nu_\th }{|z|^{1+\a-\t} d^{\t+n}(x,x+z)} \dz \dth,\\
\cT_2 & = \frac{\t}{n}\int_{\p \O_0}\int_{\T^n} \d_\xi \rho \d_z \rho \d_z u \frac{h(z) U_z \nu_\th }{|z|^{1+\a-\t} d^{\t+n}(x,x+z)} \dz \dth , \\
\cT_3 & = \frac{\t}{n}\int_{\p \O_0}\int_{\T^n} \rho \d_z \rho \d_\xi u \frac{h(z)U_z \nu_\th }{|z|^{1+\a-\t} d^{\t+n}(x,x+z)} \dz \dth , \\
\cT_4 & = \frac{\t}{n}\int_{\p \O_0}\int_{\T^n} \rho \d_\xi \rho \d_z u \frac{h(z)U_z \nu_\th }{|z|^{1+\a-\t} d^{\t+n}(x,x+z)} \dz \dth.
\end{split}
\end{equation*}
It will help us cut the estimates in half by simply noticing that the pairs $\cT_1, \cT_2$ and $\cT_3,\cT_4$ are completely similar, with the only difference in appearance of $z$ vs $\xi$ inside the finite differences. Since for any $\th$, $|\xi| \leq |z|$, and the bound \eqref{e:Sobxi1} asserts that we can replace $z$ with $\xi$ to get access to the same Sobolev norms uniformly in $\th$,  the analysis of these terms will be completely identical.  So, in what follows we only focus on $\cT_1$ and $\cT_3$.  

Looking further into the structure of $\cT_1$ and $\cT_3$ one observes that $\cT_1$ is distinctly easier in the sense that all the three state quantities appear in the form of finite differences, which allows to dissolve the singularity in $z$ among all three. Term $\cT_3$ on the other hand has only two quantities in finite difference form, while one density appears straight. For this reason the singularity presents itself relatively stronger. We will  handle this term with the help of paraproduct estimates.

\subsubsection{Estimates on $\cT_1, \cT_2$} 

Each term in the Leibnitz and Faa di Bruno expansion of $\p^m \cT_1[u,\rho]$ takes form
\[
\int_{\mathbb{T}^n}\p^l(\d_\xi \rho \d_z \rho \d_\xi u ) \prod_{i=1}^{|\bj|}\int_{\Omega(x,x+z)}\partial^{k_i} \rho(\xi)\dxi\frac{h(z)}{|z|^{1+\alpha-\tau} d^{\tau+n+ |\bj| n}(x,x+z)}\dz,
\]
where $k_1+\dots+k_{|\bj|} = m-l$. Integrating by parts in each of the topological domains and reducing to the basic $\p \O_0$ we further obtain
\[
\int_{\mathbb{T}^n}\p^l(\d_\xi \rho \d_z \rho \d_\xi u ) \prod_{i=1}^{|\bj|} \int_{\p \O_0}\d_{\xi(\th,z)}\partial^{k_i-1} \rho(x )\dth \frac{h(z)}{|z|^{1+\alpha-\tau - (n-1)|\bj|}d^{\tau+n+|\bj| n}(x,x+z)}\dz.
\]
Taking integration over communication domains outside we arrive at the term
\[
\int_{\mathbb{T}^n}\p^l(\d_\xi \rho \d_z \rho \d_\xi u )\prod_{i=1}^{|\bj|}\d_{\xi_i}\partial^{k_i-1} \rho(x) \frac{h(z)}{|z|^{1+\alpha-\tau - (n-1)|\bj|}d^{\tau+n+|\bj| n}(x,x+z)} \dz.
\]
Bounding the $d$-distance trivially we obtain the singular integral
\[
I = \int_{\mathbb{T}^n}| \p^l(\d_\xi \rho \d_z \rho \d_\xi u ) | \prod_{i=1}^{|\bj|} |\d_{\xi_i}\partial^{k_i-1} \rho(x)| \frac{h(z)}{|z|^{n+1+\alpha+|\bj|}(x,x+z)} \dz.
\]
Now, notice that the total order of derivatives in the numerator is $m - |\bj| \leq m$, the total number of terms on the top is $N = 3+ |\bj|$, and the order of singularity is $n+\a-2+N$. So, we are entirely under the scope of \lem{l:subcrit}, which gives the bound by
\[
\|I\|_2 \lesssim  \| \rho \|_{H^{m+\a}}^{\th_1} \| u\|_{H^{m+\a}}^{\th_2}, \qquad \th_1+\th_2 <1.
\]
Noticing that $m+\a < m+1+\frac{\a}{2}$, for any $\e>0$ by the generalized Young inequality,  we obtain
\begin{equation*}\label{}
\left| \int_{\T^n} \p^m e\, \p^m \cT_1[u,\rho] \dx \right | \leq C_\e Y_m + \e \| u\|_{H^{m+1 + \frac{\a}{2}}}^{2}.
\end{equation*}

Clearly, the estimate for $\cT_2$ is entirely similar.

\subsubsection{Estimates on $\cT_3, \cT_4$}
The two terms are similar, so let us focus on $\cT_3$. 
\begin{proposition}\label{}
For any $\e>0$ there is a constant $C_\e >0$ such that 
\begin{equation}\label{e:T3}
\left| \int_{\T^n} \p^m e\, \p^m \cT_3[u,\rho] \dx \right | \leq C_\e Y_m + \e \| u\|_{H^{m+1 + \frac{\a}{2}}}^{2}.
\end{equation}
\end{proposition}

Before we apply $\p^m$ on $\cT_3$, let us freeze the coefficients,
\begin{equation*}\label{}
\begin{split}
\cT_3 & = \cT_{31} + \cT_{32}, \\
\cT_{31}& = \rho^{-\t/n}(x) \int_{\T^n}  \d_z  \rho\, \d_\xi  u \frac{h(z) U_z \nu_\th }{|z|^{n+1+\a} }\dz, \\
 \cT_{32} & =  \int_{\T^n}  ( \rho\, \d_z  \rho\, \d_\xi  u) R_z \frac{h(z) U_z \nu_\th }{|z|^{n+1+\a}} \dz. \\
\end{split}
\end{equation*}

Let us first analyze $\p^m  \cT_{32}$ as this term will turn out to be subcritical and fall under the scope of \lem{l:subcrit}. 

\begin{lemma}\label{}
We have 
\[
\|\p^m  \cT_{32}\|_2 \leq \| \rho \|_{H^{m+\a}}^{\th_1} \| u\|_{H^{m+\a}}^{\th_2}, \qquad \th_1+\th_2 <1.
\]
\end{lemma}

\begin{proof} The Leibnitz expansion  consists of terms

\begin{equation*}\label{e:LexpR}
 \int_{\T^n} \p^l ( \rho\, \d_z  \rho\, \d_\xi  u) \p^{m-l} R_z \frac{h(z) U_z \nu_\th }{|z|^{n+1+\a}} \dz.
\end{equation*}
We can see that they are completely analogous to the terms we encountered in the proof of \prop{p:coerc}. The only difference is that  $N$ is smaller by one in this case. Hence, \lem{l:subcrit} applies to finish the proof.
\end{proof}

We now turn to the critical term $\cT_{31}$.

Because we have only two finite differences available, 
in order to get access to higher regularity of $\rho$ and $u$ at the first step we symmetrize in $z$ using oddness of the map $z \to U_z$: 
\begin{equation*}\label{}
\begin{split}
\cT_{31} & = \cT_{311} - \cT_{312},  \\
\cT_{311} & =  \rho^{-\t/n} \int_{\T^n}  \d^2_z  \rho\, \d_\xi  u \, \frac{h(z) U_z \nu_\th }{|z|^{n+1+\a} }\dz,\\
\cT_{312} & =  \rho^{-\t/n} \int_{\T^n}   \d_{-z}  \rho\,  \d^2_\xi  u \, \frac{h(z) U_z \nu_\th }{|z|^{n+1+\a} }\dz.
\end{split}
\end{equation*}
The two integrals are similar and have a common form 
\[
T =  \int_{\T^n}   g \ \d^2_\circ g' \, \d_\circ  g''\, \frac{h(z) U_z \nu_\th }{|z|^{n+1+\a} }\dz,
\]
where $\circ = z,-z$, or $\xi$.

\begin{lemma}\label{l:freqanal}
We have
\begin{equation*}\label{}
\int_{\T^n} f\, \p^m T  \dx  \leq  C\|f\|_2 (\|g\|_{H^{m+\a}}+ \|g'\|_{H^{m+\a}} +  \|g''\|_{H^{m+\a}}),
\end{equation*}
where $C$ depends only on the gradients of $g, g', g''$.
\end{lemma}

Applying this lemma the terms above and replacing the $m+\a$ integrability for $u$ with $m+1+\tfrac{\a}{2}$ we obtain the necessary bound \eqref{e:T3} by invoking \lem{l:powerrho}.

\begin{proof}

Moving $\p^m$ back onto $f$ we will apply the Bony decomposition to the triple $(\p^m f,g, \d_\circ g' \d_\circ g'')$, treating  the product $\d_\circ g' \d_\circ g''$ as a single piece, then splitting it into further paraproducts using \eqref{e:LL}, \eqref{e:HH}. This results in a number of terms according to interacting low (L), medium (M), and  high (H) frequencies:
\begin{equation*}
\int_{\T^n} f  \p^m T \dx  = \int_{\T^n} (\mbox{paraproducts})\\
\times  \frac{h(z) U_z \nu_\th}{|z|^{n+1+\a}} \dz,
\end{equation*}
where ``paraproducts" consists of the following terms
\begin{equation*}\label{}
\begin{split}
LMHH & = \sum_q   \sum_{p>q-1} \sum_{r> p-2} \lan \p^mf_q, g_{\sim p}, \, (\d^2_\circ g'_{\sim r} \, \d_\circ  g''_{\sim r})_{\sim p} \ran,  \\
MHLH & = \sum_q   \sum_{p>q-1} \lan \p^m f_q, g_{\sim p}, \, (\d^2_\circ g'_{<p} \, \d_\circ  g''_{\sim p})_{\sim p} \ran,  \\
MHHL & = \sum_q   \sum_{p>q-1} \lan \p^m f_q,  g_{\sim p}, \, (\d^2_\circ g'_{\sim p} \, \d_\circ  g''_{< p})_{\sim p} \ran,  \\
MLHH & = \sum_q   \sum_{r> q-2}   \lan \p^m f_q, g_{<q}, \, (\d^2_\circ g'_{\sim r} \, \d_\circ  g''_{\sim r})_{\sim q} \ran,  \\
HLLH & = \sum_q    \lan \p^m f_q,  g_{<q}, \, (\d^2_\circ g'_{<q} \, \d_\circ  g''_{\sim q})_{\sim q} \ran,  \\
HLHL & = \sum_q   \sum_{r> q-2}   \lan \p^m f_q, g_{<q}, \, (\d^2_\circ g'_{\sim q} \, \d_\circ  g''_{<q})_{\sim q} \ran, \\
HHLL & = \sum_q    \lan \p^mf_q,  g_{\sim q}, \, (\d^2_\circ g'_{<q+2} \, \d_\circ  g''_{<q+2})_{<q} \ran,  \\
LLHH & = \sum_q   \sum_{r> q+2}  \lan \p^mf_q,  g_{\sim q}, \, (\d^2_\circ g'_{\sim r} \, \d_\circ  g''_{\sim r})_{<q} \ran .
\end{split}
\end{equation*}

The general strategy will be to split the $z$-integral into short range $|z|<1/\l$ and long range $|z|>1/\l$ where $\l$ is the highest frequency of the components at hand. In the short range we use all the available gradients of $g',g''$, which is only one. 

Let us start with the low-medium-high-high term,
\begin{equation*}\label{}
\begin{split}
 \int_{\T^n}| LMHH | \frac{h(z)}{|z|^{n+1+\a}} \dz & \leq  \sum_q   \sum_{p>q-1} \sum_{r> p-2} \int_{|z| < 1/\l_r}|\lan \p^mf_q,  g_{\sim p}, \, (\d^2_\circ g'_{\sim r} \, \d_\circ  g''_{\sim r})_{\sim p} \ran|\frac{h(z)}{|z|^{n+1+\a}} \dz\\
 & + \sum_q   \sum_{p>q-1} \sum_{r> p-2} \int_{|z| >1/\l_r}|\lan \p^mf_q,  g_{\sim p}, \, (\d^2_\circ g'_{\sim r} \, \d_\circ  g''_{\sim r})_{\sim p} \ran| \frac{h(z)}{|z|^{n+1+\a}} \dz\\
&\leq  \sum_q   \sum_{p>q-1} \sum_{r> p-2} \int_{|z| < 1/\l_r} \l_q^m \|f_q\|_2 \|g_{\sim p}\|_\infty \| \n^2 g'_{\sim r} \|_2 \| \n  g''_{\sim r}\|_\infty \frac{h(z)}{|z|^{n+\a-2}} \dz\\
 & + \sum_q   \sum_{p>q-1} \sum_{r> p-2} \int_{|z| >1/\l_r}  \l_q^m \|f_q\|_2 \| g_{\sim p}\|_\infty \|  g'_{\sim r} \|_2 \| \n  g''_{\sim r}\|_\infty \frac{h(z)}{|z|^{n+\a}} \dz.
\end{split}
\end{equation*}
The first integral results in the term 
\[
\| \n^2 g'_{\sim r} \|_2 \| \n  g''_{\sim r}\|_\infty \l_r^{\a-2} \lesssim \l_r^{-m}\| \n^{m+\a} g'_{\sim r} \|_2,
\]
and  the second results in a similar term,
\[
\|  g'_{\sim r} \|_2 \| \n  g''_{\sim r}\|_\infty \l_r^{\a} \lesssim \l_r^{-m}\| \n^{m+\a} g'_{\sim r} \|_2.
\]
We continue,
\begin{equation*}\label{}
 \lesssim  \|g'\|_{H^{m+\a}} \sum_q   \sum_{p>q-1}   \frac{\l_q^m}{\l_p^m} \|f_q\|_2  \| g_{\sim p}\|_\infty \lesssim  \|g'\|_{H^{m+\a}} \|f\|_2.
\end{equation*}
In all the remaining seven terms we proceed similarly with a few modifications. Next up is MHLH,
\begin{equation*}\label{}
\begin{split}
 \int_{\T^n}| MHLH | \frac{h(z)}{|z|^{n+1+\a}} \dz & \leq  \sum_q   \sum_{p>q-1} \int_{|z| < 1/\l_p} |\lan \p^m f_q,  g_{\sim p}, \, (\d^2_\circ g'_{<p} \, \d_\circ  g''_{\sim p})_{\sim p} \ran |  \frac{h(z)}{|z|^{n+1+\a}} \dz\\
   & +  \sum_q   \sum_{p>q-1} \int_{|z| > 1/\l_p} |\lan \p^m f_q,  g_{\sim p}, \, (\d^2_\circ g'_{<p} \, \d_\circ  g''_{\sim p})_{\sim p} \ran |  \frac{h(z)}{|z|^{n+1+\a}} \dz\\
 &\leq  \sum_q   \sum_{p>q-1}  \l_q^m \|f_q\|_2 \| g_{\sim p}\|_\infty \| \n^2 g'_{<p} \|_\infty \| \n  g''_{\sim p}\|_2 \l_p^{\a-2}\\
 & + \sum_q   \sum_{p>q-1}  \l_q^m \|f_q\|_2 \| g_{\sim p}\|_\infty \| \n g'_{<p} \|_\infty \|   g''_{\sim p}\|_2\l_p^{\a}\\
 &\leq  \sum_q   \sum_{p>q-1}  \l_q^m \|f_q\|_2 \l_p^{-1} \|\n g_{\sim p}\|_\infty \| \n g'_{<p} \|_\infty \| \n^{m+\a}  g''_{\sim p}\|_2 \l_p^{-m}\\
 &\lesssim  \|g''\|_{H^{m+\a}}  \sum_q   \l_q^m \|f_q\|_2  \l_q^{-m-1}\leq \|g''\|_{H^{m+\a}} \|f\|_2.
\end{split}
\end{equation*}

Next, in a similar manner,
\begin{equation*}\label{}
\begin{split}
 \int_{\T^n}| MHHL | \frac{h(z)}{|z|^{n+1+\a}} \dz 
 &\leq  \sum_q   \sum_{p>q-1}  \l_q^m \|f_q\|_2 \| g_{\sim p}\|_\infty \| \n^2 g'_{\sim p} \|_2 \| \n  g''_{< p}\|_\infty \l_p^{\a-2}\\
 & + \sum_q   \sum_{p>q-1}  \l_q^m \|f_q\|_2 \| g_{\sim p}\|_\infty \| g'_{\sim p} \|_2 \|  \n g''_{< p}\|_\infty \l_p^{\a}\\
 &\leq  \sum_q   \sum_{p>q-1}  \l_q^m \|f_q\|_2 \l_p^{-1} \|\n g_{\sim p}\|_\infty \| \n^{m+\a} g'_{\sim p} \|_2  \l_p^{-m} \lesssim  \|g'\|_{H^{m+\a}} \|f\|_2.
\end{split}
\end{equation*}

Next,

\begin{equation*}\label{}
\begin{split}
 \int_{\T^n}| MLHH | \frac{h(z)}{|z|^{n+1+\a}} \dz & \leq  \sum_q   \sum_{r> q-2}    \int_{|z| < 1/\l_r}  |\lan \p^m f_q,  g_{<q}, \, (\d^2_\circ g'_{\sim r} \, \d_\circ  g''_{\sim r})_{\sim q} \ran | \frac{h(z)}{|z|^{n+1+\a}} \dz\\
 & + \sum_q   \sum_{r> q-2}    \int_{|z| > 1/\l_r}  |\lan \p^m f_q, g_{<q}, \, (\d^2_\circ g'_{\sim r} \, \d_\circ  g''_{\sim r})_{\sim q} \ran | \frac{h(z)}{|z|^{n+1+\a}} \dz\\
 & \leq  \sum_q   \sum_{r> q-2}      \l_q^m \| f_q\|_2 \|g_{<q}\|_\infty \| \n^2 g'_{\sim r} \|_2 \| \n  g''_{\sim r}\|_\infty \l_r^{\a - 2} \\
 & +  \sum_q   \sum_{r> q-2}     \l_q^m \| f_q\|_2 \|g_{<q}\|_\infty \|  g'_{\sim r} \|_2 \| \n  g''_{\sim r}\|_\infty \l_r^{\a} \\
 &\lesssim \sum_q   \sum_{r> q-2}   \| f_q\|_2   \frac{  \l_q^m }{\l_r^{m}} \|  \n^{m+\a} g'_{\sim r} \|_2  \lesssim  \sum_r   \|  \n^{m+\a} g'_{\sim r} \|_2 \sum_{q< r+2}   \frac{  \l_q^m }{ \l_r^{m}}  \| f_q\|_2 \\
 &\lesssim  \|g'\|_{H^{m+\a}} \left( \sum_r   \left( \sum_{q< r+2}   \frac{  \l_q^m }{ \l_r^{m}}  \| f_q\|_2 \right)^2 \right)^{1/2} \\
 &\leq  \|g'\|_{H^{m+\a}} \left( \sum_r    \sum_{q< r+2}   \frac{  \l_q^m }{ \l_r^{m}}  \| f_q\|_2^2 \right)^{1/2} \leq \|g'\|_{H^{m+\a}} \|f\|_2.
 \end{split}
\end{equation*}

In the next HLLH and HLHL terms we split relative to the scale $1/\l_q$ to obtain

\begin{multline*}\label{}
\int_{\T^n}| HLLH | \frac{h(z)}{|z|^{n+1+\a}} \dz  \leq  \sum_q      \l_q^m \| f_q\|_2 \|g_{<q}\|_\infty \| \n^2 g'_{<q} \|_\infty \| \n  g''_{\sim q}\|_2 \l_q^{\a - 2} \\
 +  \sum_q     \l_q^m \| f_q\|_2 \|g_{<q}\|_\infty \|  \n g'_{<q} \|_\infty \|  g''_{\sim q}\|_2 \l_q^{\a} \lesssim  \sum_q    \| f_q\|_2   \| \n^{m+\a} g''_{\sim q}\|_2  \leq \|f\|_2 \|g''\|_{H^{m+\a}} .
\end{multline*}


\begin{multline*}\label{}
 \int_{\T^n}| HLHL | \frac{h(z)}{|z|^{n+1+\a}} \dz  \leq  \sum_q      \l_q^m \| f_q\|_2 \|g_{<q}\|_\infty \| \n^2 g'_{
 \sim q} \|_2 \| \n  g''_{< q}\|_\infty \l_q^{\a - 2} \\
  +  \sum_q     \l_q^m \| f_q\|_2 \|g_{<q}\|_\infty \| g'_{\sim q} \|_2 \|  \n g''_{< q}\|_\infty \l_q^{\a} \lesssim  \sum_q    \| f_q\|_2   \| \n^{m+\a} g'_{\sim q}\|_2  \leq \|f\|_2 \|g'\|_{H^{m+\a}} .
\end{multline*}

Next,
\begin{equation*}\label{}
\begin{split}
 \int_{\T^n}| HHLL | \frac{h(z)}{|z|^{n+1+\a}} \dz & \leq  \sum_q      \int_{|z| < 1/\l_q}  | \lan \p^mf_q,  g_{\sim q}, \, (\d^2_\circ g'_{<q+2} \, \d_\circ  g''_{<q+2})_{<q} \ran | \frac{h(z)}{|z|^{n+1+\a}} \dz\\
 & +  \sum_q      \int_{|z| > 1/\l_q}  | \lan \p^mf_q,  g_{\sim q}, \, (\d^2_\circ g'_{<q+2} \, \d_\circ  g''_{<q+2})_{<q} \ran | \frac{h(z)}{|z|^{n+1+\a}} \dz\\
 & \leq  \sum_q      \l_q^m \| f_q\|_2 \|g_{\sim q}\|_2 \| \n^2 g'_{<q+2} \|_\infty \| \n  g''_{<q+2}\|_\infty \l_q^{\a - 2} \\
 & +  \sum_q      \l_q^m \| f_q\|_2 \|g_{\sim q}\|_2 \|  g'_{<q+2} \|_\infty \| \n  g''_{<q+2}\|_\infty \l_q^{\a } \\
 & \lesssim  \sum_q    \| f_q\|_2 \|\n^{m+\a} g_{\sim q}\|_2  \leq  \|f\|_2 \|g\|_{H^{m+\a}}.
 \end{split}
\end{equation*}

Next,
\begin{equation*}\label{}
\begin{split}
 \int_{\T^n}| LLHH | \frac{h(z)}{|z|^{n+1+\a}} \dz & \leq  \sum_q    \sum_{r> q+2}   \int_{|z| < 1/\l_r} |\lan \p^mf_q,  g_{\sim q}, \, (\d^2_\circ g'_{\sim r} \, \d_\circ  g''_{\sim r})_{<q} \ran| \frac{h(z)}{|z|^{n+1+\a}} \dz\\
 & +  \sum_q    \sum_{r> q+2}   \int_{|z| > 1/\l_r} |\lan \p^mf_q,  g_{\sim q}, \, (\d^2_\circ g'_{\sim r} \, \d_\circ  g''_{\sim r})_{<q} \ran| \frac{h(z)}{|z|^{n+1+\a}} \dz\\
 & \leq  \sum_q    \sum_{r> q+2}    \l_q^m \| f_q\|_2 \|g_{\sim q}\|_\infty \| \n^2 g'_{\sim r} \|_2 \| \n  g''_{\sim r}\|_\infty \l_r^{\a - 2} \\
 & +  \sum_q    \sum_{r> q+2}    \l_q^m \| f_q\|_2 \|g_{\sim q}\|_\infty \|  g'_{\sim r} \|_2 \| \n  g''_{\sim r}\|_\infty \l_r^{\a } \\
 & \leq  \sum_q    \sum_{r> q+2}    \l_q^m \| f_q\|_2  \| \n^{m+\a} g'_{\sim r} \|_2 \l_r^{-m}  ,
  \end{split}
\end{equation*}
and we finish as in the MLHH case.

\end{proof}

\section{Global solutions}

\subsection{Parallel shear flocks}\label{s:psf}

One immediate application of the continuation criterion would be to parallel shear flocks. Up to rotation these are given by velocities independent of $x_1$,
\begin{equation}\label{e:PSF}
u = (U(x_2,\dots,x_n,t),0,\dots,0), \qquad \rho = \rho_0(x_2,\dots,x_n).
\end{equation}
 In this case the density is stationary $\p_t \rho = 0$, and so is the $e$-quantity. The momentum equation takes the form of  pure diffusion 
\begin{equation}\label{e:parabolic}
U_t(x) = \int_{\T^n} \phi(x,x+z) \rho_0(x+z) \d_z U(x) \dz.
\end{equation}
Taking partial derivative $\p$ with respect to any variable results in
\[
\p U_t =  \int_{\T^n} \phi(x,x+z) \rho_0(x+z) \d_z \p U(x) \dz +  \int_{\T^n} \p_x [\phi(x,x+z) \rho_0(x+z)] \d_z U(x) \dz.
\]
If $\rho_0 \in H^{m+\a}$ and $\rho_0(x) \geq \rmin >0$, then  $\p_x [\phi(x,x+z) \rho_0(x+z)] $ is still a kernel of singularity $n+\a$. Thus, if $\a<1$, the last integral is bounded by $C \| \n U\|_\infty$.  Evaluating the above at the point of maximum of $\p U$ and summing over all partials gives the inequality
\[
\ddt \| \n U\|_\infty \leq C \| \n U\|_\infty.
\]
It is therefore a priori bounded and the criterion applies.

To handle the range $1\leq \a<2$, we note that \eqref{e:parabolic} fall under a general class of fractional diffusion equations
\[
w_t(x) =\int_{\T^n} K(x,z,t)\d_z w(x) \dz.
\]
Typical regularity results for such equations requires either the assumption of evenness of $K$ in $z$ or symmetry in $x,y=x+z$. In our settings, the kernel is given by
\begin{equation*}\label{e:kernelKk}
K=\frac{h(z)}{|z|^{n+\a}} \k(x,x+z), \qquad \text{with} \qquad \k(x,x+z)=\frac{\rho_0(x+z)}{\left(\fint_{\O(x,x+z)}\rho_0(\xi)\dxi \right)^{\frac{\t}{n}}},
\end{equation*}
which satisfies neither of the above requirements. However, freezing the coefficients and representing the kernel as a sum of the main even and residual parts
where
\begin{equation*}\label{e:kernelF&G}
F(x,z)=\frac{h(z)}{|z|^{n+1}}\k(x,x), \qquad G(x,z)=\frac{h(z)}{|z|^{n+1}}\left(\k(x,x+z)-\k(x,x)\right),
\end{equation*}
fulfills the hypotheses of Theorem 1.1 of \cite{IJS2018} in the particular case $\a=1$, which provides Schauder estimates
\[
\|U\|_{C^{1+\g}(\T^n\times[T/2,T))}\lesssim \|U\|_{L^{\infty}(\T^n\times[0,T))}.
\]
Since the right hand side is uniformly bounded due to the maximum principle, this  fulfills the continuation criterion and
the proof  is complete for $\a = 1$.

For the subcritical case $\a>1$ we can adopt \cite[Theorem 8.1]{SS2016}. To get the model to satisfy its assumptions, we use a cut-off function $\chi(z)$ and small $\e>0$ to break the residual term $G$ further into inner singular part
\[
G_\e = \frac{h(z)\chi(z/\e)}{|z|^{n+\a}}\left[\frac{\rho_0(x+z)}{\left(\fint_{\O(0,z)}\rho_0(x+\xi)\dxi \right)^{\frac{\t}{n}}} - \frac{\rho_0(x)}{\rho_0(x)^{\frac{\t}{n}}}\right],
\]
and the outer regular part
\[
H_\e =  \frac{h(z)(1-\chi(z/\e))}{|z|^{n+\a}}\left[\frac{\rho_0(x+z)}{\left(\fint_{\O(0,z)}\rho_0(x+\xi)\dxi \right)^{\frac{\t}{n}}} - \frac{\rho_0(x)}{\rho_0(x)^{\frac{\t}{n}}}\right].
\]
The integral 
\[
f = \int H_\e(x,z) \d_z U(x)\dz,
\]
contributes with a bounded source to the equation, while the principal kernel
\[
K_\e = F + G_\e,
\]
for small $\e>0$ satisfies all the assumptions (A1)--(A4) of \cite{SS2016}. 

For $\a>1$ it is necessary to use the next Taylor term in the definition of the finite difference, so we add and subtract it in the singular integral to produce an extra drift:
\begin{equation}\label{e:m2}
U_t + b \cdot \n U = \int K_\e(x,z) [\d_z U(x) - z \cdot \n U(x) ] \dz + f,
\end{equation}
where
\[
b (x) = - \int_{\T^n}  G_\e(x,z) z \dz.
\]
The latter is no longer a singular integral and hence $b$ is bounded. This fulfills all the assumptions of \cite[Theorem 8.1]{SS2016} pertaining to the equation \eqref{e:m2}, and the $C^{1+\g}$-continuity of $u$ follows.  We have proved the following theorem.

\begin{theorem}\label{}
For any $U_0, \rho_0 \in H^{m+1} \times H^{m+\a}$, $0<\a<2$, there exists a unique global solution to the equation \eqref{e:main}--\eqref{e:kernel} in the form of a parallel shear flock \eqref{e:PSF} which belongs to the same class.
\end{theorem}

\subsection{Nearly aligned flocks}\label{s:naf}
In this section we reveal another class of global solutions with nearly aligned velocity field. We denote homogeneous H\"older norms by
\[
[f]_k = \|\n^k f\|_\infty.
\]

\begin{theorem}\label{t:naf}
There exists an $R_0 >0$ and $N \in \N$ dependent only on the parameters of the system and $m$ such that if $R>R_0$ and the initial condition satisfies
\[
Y_m(0) + \rmin^{-1}_0 + \rmax_0 \leq R, \qquad \cA_0 \leq \frac{1}{R^N},
\]
then there exists a global unique solution to \eqref{e:main} starting from such initial condition. Moreover, such solution will align exponentially fast,
\begin{equation*}\label{e:AexpRt}
\cA(t) \leq \frac{1}{R^{N-b}}  e^{- \frac{c}{R^a} t} ,
\end{equation*}
where $a,b>0$ depend only on the parameters of the system, and flock to a smooth traveling density profile $\rho_\infty$:
\begin{equation}\label{e:sf}
\rho \to \rho_\infty(x- t \bar{u}).
\end{equation}

\end{theorem}
\begin{proof}

We assume without loss of generality that the momentum vanishes, $P = 0$. 

Let us observe that since $\|u(t)\|_{H^{m+1}}$ on a given time interval controls $[\rho]_1$ via the continuity equation,  the local solution to \eqref{e:main} according to our criterion can be extended up to the critical time $t = t^*$, at which 
\[
\|u(t^*)\|_{H^{m+1}} = 2R,
\]
for the first time, i.e.
\[
\|u(t)\|_{H^{m+1}} < 2R, \qquad t<t^*.
\]
Our goal will be to show that such time $t^*$ never happens, and thus the solution is global.

\begin{lemma}\label{}
On the same time interval $[0,t^*]$ we have
\begin{equation}\label{e:rR}
\frac{1}{2R} \leq \rmin \leq \rmax \leq 2R.
\end{equation}
\end{lemma}
\begin{proof}
Let $t^{**}$ be the first time when one of these inequalities fails. We would like to show that $t^{**} > t^*$. If not, let us make a preliminary alignment estimate on the shorter time interval $[0,t^{**}]$.

Recall the energy law \eqref{e:energy}. Note that
\[
\phi(x,y) \geq \frac{\l}{M^{\t/n}} \one_{|x-y|<r_0}.
\]
So, we have
\[
\ddt \cE \leq - \frac{c}{R^{2+\frac{\t}{n}}} \int_{|x-y|<r_0} | u(x) - u(y)|^2 \dy \dx.
\]
Applying \cite[Lemma 2.1]{LSlimiting}, we further continue
\[
\ddt \cE \leq -\frac{c}{R^{2+\frac{\t}{n}}} \int_{\T^n} |u(x) - \bar{u}|^2 \dx,
\]
where $\bar{u}$ is the usual average of $u$. Note that it may not be $0$ despite vanishing momentum. Let us reinsert the density noting that $\rho/R \leq 2$,
\[
\ddt \cE \leq -\frac{c}{R^{3+\frac{\t}{n}}} \int_{\T^n} \rho(x) |u(x) - \bar{u}|^2 \dx.
\]
Expanding the square and using vanishing of the momentum we obtain
\[
\ddt \cE \leq - \frac{c}{R^{3+\frac{\t}{n}}} \int_{\T^n} \rho(x) |u(x)|^2 \dx = - \frac{c}{R^{3+\frac{\t}{n}}}  \cE.
\]
Thus, on the time interval $[0,t^{**}]$ we have 
\begin{equation*}\label{ }
\cE(t) \leq \cE_0 e^{- \frac{c}{R^a} t}, \qquad a = 3+\frac{\t}{n}.
\end{equation*}
Let us note that initial energy is bounded by (again using the vanishing momentum)
\[
\cE_0 \leq  \int_{\T^{2n}} \rho_0(x)\rho_0(y) | u_0(x) - u_0(y)|^2 \dy \dx \leq \cA_0 M^2 \leq C R^{2-N}.
\]
So,
\[
\cE(t) \leq C R^{2-N}e^{- \frac{c}{R^a} t}.
\]

Now we estimate the decay of the amplitude $\cA$ itself. Let us pick one coordinate of $u$ and evaluate at a point of its maximum $x_+$:
\begin{equation*}\label{}
\begin{split}
\ddt u(x_+) & \leq   \frac{c}{R^{\t/n}} \int_{|z| < r_0} \rho(x_+ +z) [u(x_++z) - u(x_+)] \dz \leq  \frac{c}{R^{\t/n}} M^{1/2} \cE^{1/2} -  \frac{c}{R^{\t/n}} u(x_+) \\
&\leq C R^{5/2 - \t/n - N} e^{- \frac{c}{R^a} t} -  \frac{c}{R^{\t/n}} u(x_+).
\end{split}
\end{equation*}
Similarly,
\[
\ddt u(x_-) \geq  - C R^{5/2 - \t/n - N} e^{- \frac{c}{R^a} t} +  \frac{c}{R^{\t/n}} u(x_-).
\]
Subtracting the two,
\[
\ddt \cA \leq C R^{5/2 - \t/n - N} e^{- \frac{c}{R^a} t} -  \frac{c}{R^{\t/n}} \cA.
\]
By the Gr\"onwall inquality,
\[
\cA(t) \leq \cA_0 e^{- \frac{c}{R^{\t/n} }t} + C R^{5/2 - \t/n - N} t e^{- \frac{c}{R^a} t} \leq \frac{1}{R^N}  e^{- \frac{c}{R^a} t} +C R^{5/2 - \t/n - N+a} e^{- \frac{c}{R^a} t}.
\]
Thus,
\begin{equation}\label{e:AexpR}
\cA(t) \leq \frac{1}{R^{N-b}}  e^{- \frac{c}{R^a} t} ,
\end{equation}
where $b$ depends only on $n, \t$. By interpolation,
\[
\| \n u\|_\infty \leq \cA^{\th} \| u\|_{H^{m+1}}^{1-\th} \leq \frac{1}{R^{\th(N-b) + \th - 1}}  e^{- \frac{\th c}{R^a} t} .
\]
Integrating the continuity equation along characteristics,  we obtain
\[
\rmax \leq \| \rho_0 \|_\infty \exp\left\{ \int_0^t \| \n u\|_\infty \ds \right\} \leq R \exp \left\{ \frac{1}{\th c R^{\th(N-b) + \th - 1 - a }} \right\},
\]
and similarly,
\[
\rmin \geq \frac{1}{R} \exp \left\{ - \frac{1}{\th c R^{\th(N-b) + \th - 1 - a }} \right\}.
\]
Clearly, if $R$ and $N$ are large enough the exponential is $<2$. This leads to a contradiction with the definition of $t^{**}$.
\end{proof}

From this point on we will denote by $\cN(t)$ any ``negligent" quantity which has a bound of the form
\[
\cN(t) \leq \frac{C}{R^{\th N}} e^{ -  \frac{c}{R^a} t },
\]
where $C,c>0$, and $0<\th<1$ and $a>0$ depend only on $\t$, $n$, $\a$, $m$, the parameters of the system.
We observe the identities
\begin{equation*}\label{}
R^b \cN \sim \cN, \quad  \cN^\l \sim \cN, \quad \mbox{etc}.
\end{equation*}

As a consequence of the proof of the lemma we have shown that as long as \eqref{e:rR} holds, the estimate on the amplitude \eqref{e:AexpR} holds. As a result, by interpolation, we have similar exponential bounds in H\"older classes,
\begin{equation*}\label{}
[u]_1, [u]_2, [u]_3 \leq \cN,
\end{equation*}
and as a consequence, from the continuity equation, 
\begin{equation*}\label{}
[\rho]_1, [\rho]_2 \leq 2R, 
\end{equation*}
for all $t<t^*$, provided $N$ is large enough. As a further consequence, we obtain
\[
\|e\|_\infty \leq [u]_1+ \| \cL_{\phi} \rho \|_\infty.
\]
We appeal to \cite[Lemma B.1]{STtopo} (with $r=1$ and $\g=0$) to conclude 
\[
\| \cL_\phi \rho \|_\infty \leq  R^b ([\rho]_2 + \|\rho\|_\infty + [\rho]_1^2) \leq R^b,
\]
where $b$ depends only on the parameters of the system. Thus,
\begin{equation}\label{e:eRinfty}
\|e(t)\|_\infty \leq C R^b, \qquad t<t^*.
\end{equation}
Now, if we look back at the $e$-equation, we can see that all the transport terms contain a power of $[u]_1$, and with \eqref{e:eRinfty}  can be estimated by $\leq \cN Y_m$. At the same time the topological terms $\cT_1,\cT_2$ will include a power of $[u]_1$ as a result of application of \lem{l:subcrit},
\[
\|\cT_{1,2}\| \leq \cN \| \rho \|_{H^{m+\a}}^{\th_1} \| u\|_{H^{m+\a}}^{\th_2} \leq \cN Y_m + \cN+ \frac{\e}{R^{1+\t/n}} \|u\|^2_{H^{m+1+\frac{\a}{2}}}. 
\]
Similarly, all paraproduct estimates of \lem{l:freqanal} for  $\cT_3,\cT_4$ will include either $ \| u\|_{H^{m+\a}}$ or $\| \rho \|_{H^{m+\a}}$ coupled with a power of $[u]_1$. Again, in the latter case this results in a factor of $\cN$, while in the former case, by interpolation 
\[
 \| u\|_{H^{m+\a}} \leq \cN \|u\|_{H^{m+1+\frac{\a}{2}}}^\th.
\]
In summary we have a factor of $\cN$ to appear in the main term of the $e$-equation:
\begin{equation*}\label{e:eR}
\ddt \|e\|_{H^m}^2 \leq \cN Y_m + \cN+ \frac{\e}{R^{1+\t/n}} \|u\|^2_{H^{m+1+\frac{\a}{2}}}. 
\end{equation*}

Examining the $u$-equation in a similar manner we conclude
\begin{equation}\label{e:uRY}
\ddt \|u\|_{H^{m+1}}^2 \leq \cN Y_m + \cN - \frac{c_0}{R^{1+\t/n}} \|u\|^2_{H^{m+1+\frac{\a}{2}}}. 
\end{equation}
Adding the two together, we obtain
\[
\ddt(Y_m+1) \leq \cN (Y_m +1) - \frac{c_0}{2R^{1+\t/n}} \|u\|^2_{H^{m+1+\frac{\a}{2}}}. 
\]
Ignoring the dissipation term for a moment we conclude by integration that 
\[
(Y_m+1)(t^*) \leq (R+1) e^{R^{-\th N}} \leq 2R,
\]
if $N$ and $R$ are large enough.  Thus,
\[
Y_m(t^*) \leq 3R.
\]
Plugging this back into \eqref{e:uRY} we conclude that at the critical time $t^*$, 
\[
\ddt \|u\|_{H^{m+1}}^2 \leq 3 \cN  R + \cN - c_0 R^{1 - \tau/n} \leq 4 R^{1-\th N} - c_0 R^{1 - \tau/n} < 0,
\]
if again $R$ and $N$ are chosen large enough. This is a contradiction with the definition of $t^*$.

The flock convergence \eqref{e:sf} follows immediately from the continuity equation and exponential decay of all norms of $u$ up to $H^{m+1}$.
\end{proof}

\bibliographystyle{plain}
\bibliography{CCtopo,collective-appl,collective-pure}

\end{document}